\documentclass[11pt]{amsart}
\usepackage{amssymb,amsmath,amsthm,enumitem,colonequals,mlmodern,tikz-cd,microtype}
\usepackage[cal=euler,bfcal,bb=px,bfbb]{mathalpha}
\usepackage{epsfig}
\usetikzlibrary{positioning}
\usepackage{url}
\usepackage{tcolorbox}
\usepackage{setspace}
\usepackage{fancyhdr}
\usepackage{pdfpages}
\usepackage{color}
\usepackage{verbatim}
\usepackage{fancyhdr}
\usepackage[top=2.5cm, bottom=2.5cm, left=2cm, right=2cm]{geometry}
\usepackage{xcolor}
\colorlet{darkblue}{blue!55!black}
\colorlet{darkcyan}{cyan!50!black}
\colorlet{darkgreen}{green!60!black} 

\PassOptionsToPackage{hyphens}{url}

\def\eqref#1{\textcolor{darkblue}{(\ref{#1})}}

\usepackage[nameinlink]{cleveref} 
\Crefformat{section}{#2\S#1#3}
\Crefmultiformat{section}{#2\S\S#1#3}{ and~#2#1#3}{, #2#1#3}{, and~#2#1#3}

\usepackage[pagewise]{lineno}
\let\oldequation\equation
\let\oldendequation\endequation
\renewenvironment{equation}{\linenomathNonumbers\oldequation}{\oldendequation\endlinenomath}
\expandafter\let\expandafter\oldequationstar\csname equation*\endcsname
\expandafter\let\expandafter\oldendequationstar\csname endequation*\endcsname
\renewenvironment{equation*}{\linenomathNonumbers\oldequationstar}{\oldendequationstar\endlinenomath}
\let\oldalign\align
\let\oldendalign\endalign
\renewenvironment{align}{\linenomathNonumbers\oldalign}{\oldendalign\endlinenomath}
\expandafter\let\expandafter\oldalignstar\csname align*\endcsname
\expandafter\let\expandafter\oldendalignstar\csname endalign*\endcsname
\renewenvironment{align*}{\linenomathNonumbers\oldalignstar}{\oldendalignstar\endlinenomath}

\newcommand {\Hom} {\mathsf{Hom}}
\newcommand {\End} {\mathcal E\mathsf{nd}}

\newcommand  {\rank}     {\textsf{rank}}

\DeclareMathOperator{\Tot}{Tot}

\DeclareMathOperator{\Aut}{Aut}

\theoremstyle{plain}
\newtheorem{theorem}{Theorem}[section]
\newtheorem{lemma}[theorem]{Lemma}
\newtheorem{corollary}[theorem]{Corollary}
\newtheorem{proposition}[theorem]{Proposition}

\theoremstyle{definition}
\newtheorem{definition}[theorem]{Definition}

\newtheorem{remark}[theorem]{Remark}

\newtheorem{notation}[theorem]{Notation}

\newtheorem*{ack}{Acknowledgments}

\AddToHook{env/conjecture/begin}{\crefalias{theorem}{conjecture}}
\AddToHook{env/lemma/begin}{\crefalias{theorem}{lemma}}
\AddToHook{env/corollary/begin}{\crefalias{theorem}{corollary}}
\AddToHook{env/proposition/begin}{\crefalias{theorem}{proposition}}
\AddToHook{env/definition/begin}{\crefalias{theorem}{definition}}
\AddToHook{env/remark/begin}{\crefalias{theorem}{remark}}
\AddToHook{env/example/begin}{\crefalias{theorem}{example}}
\AddToHook{env/hypothesis/begin}{\crefalias{theorem}{hypothesis}}
\AddToHook{env/notation/begin}{\crefalias{theorem}{notation}}

\numberwithin{equation}{section}
\numberwithin{theorem}{section}

\title{Infinitesimal Deformations of Generalized Parabolic Hitchin Pairs}
\author{Sourav Das}
\email{sdas6565@gmail.com}
\date{\today}

\begin{document}
\maketitle

\begin{abstract}
We develop the infinitesimal deformation theory of generalized parabolic Hitchin pairs (GPHs) on irreducible nodal curves. For any GPH $(E,\phi,F(E))$ we associate an explicit three-term deformation complex $\mathcal{C}_{(E_\bullet,\phi)}$ whose first hypercohomology $\mathbb{H}^1(\mathcal{C}_{(E_\bullet,\phi)})$ parameterizes first-order deformations.

As the main application, we prove that the coarse moduli space $\mathcal{M}_{\mathrm{GPH}}$ of $1$-stable generalized parabolic Hitchin pairs of rank $n$ and degree $d$ with $\gcd(n,d)=1$ is a local complete intersection, smooth in codimension one, and hence a normal variety. We further show that this moduli space carries a natural Poisson structure.
\end{abstract}

\tableofcontents

\section{\textbf{Introduction}}

 The study of moduli spaces of vector bundles and their generalizations on singular curves has been a central theme in algebraic geometry for several decades \cite{zbMATH06869614, Logares2019}. A particularly fruitful approach to understanding torsion-free sheaves on nodal curves was pioneered by U.~N.~Bhosle, who introduced the notion of \emph{generalized parabolic bundles} (GPBs) in \cite{Bh1}. These objects provide a natural desingularization of the moduli space of torsion-free sheaves on a nodal curve in the coprime rank--degree case. The construction has found numerous deep applications: these ideas played a key role in the proof of the factorization theorem for Verlinde numbers by Ramadas and Narasimhan \cite{NR, R}, appears prominently in Seshadri's ICTP notes \cite{ICTP}, and has been used extensively in the study of degenerations of moduli spaces of vector bundles on smooth curves \cite{Ka, Sun, BBN}. It is worth mentioning that generalized parabolic structures also play a role in the study of principal bundles on nodal curves, as clarified in the work of Schmitt \cite{Schmitt} and Mu\~noz Casta\~neda \cite{Munoz-GP, Munoz-singular, Munoz-Schmitt}.

Motivated by these developments, Bhosle extended the theory to \emph{generalized parabolic Hitchin pairs} (GPHs) in \cite{Bh2}. A GPH consists of a vector bundle \(E\) on the normalization \(\widetilde{X}_0\) of a nodal curve \(X_0\), equipped with a Higgs field \(\phi\) compatible with a generalized parabolic structure at the preimages of the node. Bhosle constructed the moduli space of GPHs and described the fibres of the associated Hitchin map \cite{Bh2}. The properness of this Hitchin map, together with a detailed comparison with a natural semistable degeneration of the moduli of Higgs bundles constructed by Balaji-Barik-Nagaraj \cite{BBN}, was later established by Das in \cite{D}.

The moduli theory of Higgs bundles and semistable pairs was laid down by the foundational work of Hitchin \cite{HI, HII} and Nitsure \cite{Nit}. This theory has since been extended to nodal and, more generally, singular curves: Balaji, Barik, and Nagaraj \cite{BBN} constructed a semistable degeneration of the moduli space of Higgs bundles, and Ben-Bassat, Das, and Pantev \cite{BDP} constructed a comprehensive framework for moduli stacks of Higgs bundles on stable curves. Poisson structures on the associated moduli spaces have also been studied in \cite{Kydonakis, DI}.

 In this article, we carry out a systematic study of the \emph{infinitesimal deformation theory} of generalized parabolic Hitchin pairs, a theory that has, until now, remained absent from the literature. Such a theory is essential for understanding the local structure of the moduli space (its singularities and smoothness), for establishing global geometric properties such as normality and the local complete intersection property, and for comparing different compactifications and degenerations of Higgs bundle moduli spaces.

 For a smooth projective curve, irreducibility and normality of the moduli space of semistable Higgs bundles is comparatively well understood. Hitchin \cite{HI, HII} constructed the moduli space of rank-two Higgs bundles as a smooth complex manifold, and deduced connectedness, hence irreducibility. Nitsure \cite{Nit} extended the construction to arbitrary rank as a quasi-projective scheme and established smoothness at points corresponding to stable pairs via an explicit deformation complex; since $\gcd(n,d)=1$ forces every semistable point to be stable, the whole moduli space is smooth in this case. Alternatively, normality on the Betti side transfers to the Dolbeault (Higgs bundle) side via Simpson's \emph{equisingularity principle} \cite{Simpson2}: since the differential graded Lie algebras governing the deformation theories of a flat connection and its corresponding Higgs bundle are formal and share the same cohomology (Goldman--Millson \cite{GM}), the germs of the de Rham/Betti and Dolbeault moduli spaces at corresponding semisimple points are analytically isomorphic. Simpson establishes normality on the Betti side - for representations of a genus-$\geq 2$ surface group of degree $0$ - by showing it is a local complete intersection and invoking Serre's criterion, and this transfers to the Dolbeault side by the equisingularity principle.

Before returning to the discussion of moduli of Higgs bundles (torsion-free sheaves) on nodal curves: for the (Higgs-field-free) moduli space of semistable torsion-free sheaves on a nodal curve, the parameter scheme is not smooth, or even normal, at every point -- its local structure depends on the local type of the sheaf at the node. Seshadri \cite{CS}, and later Faltings \cite{Fal} by a different method, determined the analytic local ring of this moduli space at a stable sheaf $F$ whose local type at the node is the most degenerate one, $F_p \cong \mathfrak m^{\oplus n}$ (where $\mathfrak m$ is the maximal ideal of $\mathcal O_{X_0,p}$), showing that it is formally smooth over $\mathbb C[[X,Y]]/(XY-YX)$, the ring governing pairs of commuting $n\times n$ matrices. Seshadri then reduced the question of reducedness of the moduli space itself to that of this commuting-variety local model, which was settled by Cowsik for $n=2$ and by Strickland, using Schubert calculus, for general $n$. Only with this explicit local analysis in hand could reducedness, and eventually Cohen--Macaulayness, of the moduli space of torsion-free sheaves on a nodal curve be established. It is worth recalling here that Bhosle's moduli of GPBs provides a desingularization of the moduli of torsion-free sheaves.

\subsubsection{The difficulties of our problem}
The Higgs case on a nodal curve is somewhat more complicated. The analogous Quot scheme for Higgs pairs is the spectrum of the symmetric algebra of a non-locally-free sheaf over the Quot scheme of torsion-free sheaves. In this article we do not study this moduli problem; instead, we study the moduli of GPHs on the normalization of the nodal curve. Even so, the Quot scheme of GPHs is constructed as the vanishing locus of a section of a certain universal vector bundle over the Quot scheme of vector bundles on the normalization \cite[Appendix, PhD Thesis]{D}. It is therefore not immediate that this Quot scheme is normal, and this in turn requires the explicit deformation-theoretic argument that this article develops.
 
 \subsubsection{Constribution of this work} We introduce an explicit three-term deformation complex \(\mathcal{C}_{(E_\bullet,\phi)}\) whose first hypercohomology \(\mathbb{H}^1(\mathcal{C}_{(E_\bullet,\phi)})\) parametrizes the first-order infinitesimal deformations of a GPH \((E,\phi,F(E))\) (Theorem~\ref{thm:def-GPH}). Using short exact sequences of complexes and the hypercohomology spectral sequence (developed in the appendix \ref{Append}), we perform detailed computations of the obstruction spaces. In particular, when the underlying vector bundle is stable, we prove that \(\dim \mathbb{H}^2(\mathcal{C}_{(E_\bullet,\phi)}) = 1\) in two key cases: when the generalized parabolic structure has full rank (\(\operatorname{rank} A = n\)) and when \(\operatorname{rank} A = n-1\) and the underlying bundle is stable (Section~\ref{sec:computations}).

As the main geometric application, we show that the coarse moduli space of stable generalized parabolic Hitchin pairs of rank \(n\) and degree \(d\) with \(\gcd(n,d)=1\) is a local complete intersection that is smooth in codimension one, and is therefore a normal variety (Theorem~\ref{NT}).

This has immediate consequences. There exists a natural log-symplectic form on the smooth locus of \(\mathcal{M}_{\mathrm{GPH}}\) (outside a codimension-two subset). Since $\mathcal M_{\mathrm{GPH}}$ is normal, this form extends as a Poisson structure on the whole moduli space. 

\subsubsection{Statements of the main results} We now state the main results of the paper. Let \(X_0\) be an irreducible projective nodal curve of genus \(g \geq 2\) with exactly one node \(p\). Let \(\pi \colon \widetilde{X}_0 \to X_0\) be its normalization. Then \(\widetilde{X}_0\) is a smooth projective curve of genus \(g-1\), and \(\pi^{-1}(p) = \{x^+, x^-\}\) consists of two distinct points. We will sometimes denote the divisor by $D:=\{x^+, x^-\}$. Also, for a vector bundle $E$ on $\widetilde{X_0}$, we denote by $E_D:=E_{x^+}\oplus E_{x^-}$.

\begin{theorem}[Deformation complex for GPHs]
The isomorphism classes of first-order infinitesimal deformations of a generalized parabolic Hitchin pair \((E_{\bullet},\phi):=(E,\phi, F(E))\) are canonically parametrized by the first hypercohomology \(\mathbb{H}^1(\mathcal{C}_{(E_{\bullet}, \phi)})\), where \(\mathcal{C}_{(E_{\bullet}, \phi)}\) is the three-term complex of sheaves on the nodal curve \(X_0\):
\[
\mathcal{C}_{(E_{\bullet}, \phi)} \colon \
\pi_*\End(E) \ \xrightarrow{d^0}\
\pi_*\End(E)\otimes\omega_{X_0} \oplus \Hom(F,E_D/F)
\ \xrightarrow{d^1}\
\Hom(F,E_D/F).
\]
The differentials \(d^0\) and \(d^1\) are given explicitly in Theorem~\ref{thm:def-GPH}.
\end{theorem}

The above theorem extends straightforwardly to any nodal curve with an arbitrary number of nodes. The only modifications required are local at each node; the global arguments remain unchanged.

\begin{lemma}[Obstruction space]
Let \((E, F(E):=\Gamma_A, \phi)\) be a generalized parabolic Higgs bundle on the nodal curve \(X_0\) such that $A: E_{x^+}\to E_{x^-}$ is a map of vector spaces and $\Gamma_A$ is the graph of $A$ and such that the underlying bundle $E$ is stable.
\begin{enumerate}
\item If \(\operatorname{rank} A = n\) (full rank gluing), then \(\dim \mathbb{H}^2(\mathcal{C}_{(E_{\bullet},\phi)}) = 1\).
\item If \(\operatorname{rank} A = n-1\), then \(\dim \mathbb{H}^2(\mathcal{C}_{(E_{\bullet},\phi)}) = 1\).
\end{enumerate}
\end{lemma}

\begin{theorem}[Normality]
Let \(\mathcal{M}_{\mathrm{GPH}}\) be the coarse moduli space of stable generalized parabolic Hitchin pairs of rank \(n\) and degree \(d\) with \(\gcd(n,d)=1\). Then \(\mathcal{M}_{\mathrm{GPH}}\) is a local complete intersection that is smooth in codimension one, and is therefore a normal variety.
\end{theorem}

\begin{corollary}
The moduli space \(\mathcal{M}_{\mathrm{GPH}}\) carries a natural Poisson structure extending the logarithmic-symplectic form defined on its smooth locus (outside a codimension-two subset).
\end{corollary}

\textbf{Open Questions.} 
This work builds the basic ideas and structures for a detailed study of the Poisson geometry of the moduli space $\mathcal{M}_{\mathrm{GPH}}$. In particular, it would be interesting to investigate the symplectic leaves and the degeneracy locus of the Poisson structure. 

Furthermore, a systematic analysis of the singularities of $\mathcal{M}_{\mathrm{GPH}}$ remains an open question. Building on the ideas in Seshadri's ICTP notes~\cite{ICTP}, Kausz's work \cite{Ka} and the work of Balaji--Barik--Nagaraj~\cite{BBN}, we expect that the moduli space of generalized parabolic Hitchin pairs will serve as a powerful tool for computing cohomological invariants and understanding the geometry of Gieseker-type degenerations of the moduli space of Higgs bundles on smooth curves (see \cite{KiemLi2007}). 

\subsection{Notation and Conventions}

We collect the main notation concerning the moduli objects appearing in this paper.

\begin{tcolorbox}[colback=white, sharp corners]
\centering
\small
\begin{tabular}{l|l|l|l}
\textbf{Object} & \textbf{Functor} & \textbf{Stack} & \textbf{Coarse space} \\
\hline
Gen.\ parabolic Hitchin pairs & $\mathrm{GPH}$ & $\mathrm{M}_{\mathrm{GPH}}$ & $\mathcal{M}_{\mathrm{GPH}}$ \\
Gieseker--Higgs bundles & $\mathrm{GHB}$ & $\mathrm{M}_{\mathrm{GHB}}$ & $\mathcal{M}_{\mathrm{GHB}}$ \\
Torsion-free sheaves & $\mathrm{TF}$ & $\mathrm{M}_{\mathrm{TF}}$ & $\mathcal{M}_{\mathrm{TF}}$ \\
Torsion-free Higgs sheaves & $\mathrm{TFH}$ & $\mathrm{M}_{\mathrm{TFH}}$ & $\mathcal{M}_{\mathrm{TFH}}$ \\
Gen.\ parabolic bundles & $\mathrm{GPB}$ & $\mathrm{M}_{\mathrm{GPB}}$ & $\mathcal{M}_{\mathrm{GPB}}$ \\
Hitchin pairs (normalization) & $\mathrm{HP}$ & $\mathrm{M}_{\mathrm{HP}}$ & $\mathcal{M}_{\mathrm{HP}}$ \\
Vector bundles (normalization) & $\mathrm{Bun}$ & $\mathrm{M}_{\mathrm{Bun}}$ & $\mathcal{M}_{\mathrm{Bun}}$ \\
Gieseker vector bundles & $\mathrm{GVB}$ & $\mathrm{M}_{\mathrm{GVB}}$ & $\mathcal{M}_{\mathrm{GVB}}$ \\
\end{tabular}

\medskip
Level-one versions are denoted with a superscript $1$ (e.g.\ $\mathrm{GHB}^1$, $\mathrm{M}^1_{\mathrm{GHB}}$, $\mathcal{M}^1_{\mathrm{GHB}}$).
\end{tcolorbox}

\begin{ack}
The author thanks the University of Haifa for its support and hospitality, where part of this work was carried out while he was a postdoctoral fellow. He also thanks Grok AI for suggesting the use of the spectral sequence technique employed in \cref{sec:computations}, and for its assistance in improving the exposition of this paper.
\end{ack}

\section{\textbf{The first order infinitesimal deformations of GPHs}}

Let \(X_0\) be an irreducible projective nodal curve of genus \(g \geq 2\) with exactly one node \(p\). Let \(\pi \colon \widetilde{X}_0 \to X_0\) be its normalization. Then \(\widetilde{X}_0\) is a smooth projective curve of genus \(g-1\), and \(\pi^{-1}(p) = \{x^+, x^-\}\) consists of two distinct points. We will sometimes denote the divisor by $D:=\{x^+, x^-\}$.

\begin{definition}\cite[Definition 2.1.]{Bh2}
A \emph{generalized parabolic Hitchin pair} (GPH) on \(X_0\) is a triple \((E, \phi, F(E))\), where
\begin{enumerate}
\item \(E\) is a vector bundle of rank \(n\) and degree \(d\) on the normalization \(\widetilde{X}_0\),
\item \(\phi \colon E \to E \otimes \pi^*\omega_{X_0}\) is a Higgs field (i.e., an \(\mathcal{O}_{\widetilde{X}_0}\)-linear map), and $\omega_{X_0}$ is the dualising sheaf of the nodal curve $X_0$,
\item \(F(E) \subset E_{x^+} \oplus E_{x^-}\) is a linear subspace of dimension \(n\) (a generalized parabolic structure),
\end{enumerate}
such that \(\phi\) preserves the generalized parabolic structure, i.e.,
\[
(\pi_*\phi)(F(E)) \subset F(E) \otimes \omega_{X_0}.
\]
We denote such a GPH by \((E_{\bullet}, \phi) = (E, \phi, F(E))\).
\end{definition}

\begin{notation}
For a vector bundle $E$ on $\widetilde X_0$ we write
\[
E_D \ \colonequals\ E_{x^+}\oplus E_{x^-}
\]
for the direct sum of its fibres over the two preimages of the node. More generally, for a subbundle $E'\subseteq E$ we write $E'_D \colonequals E'_{x^+}\oplus E'_{x^-} \subseteq E_D$.
\end{notation}

\begin{remark}[The dualizing sheaf of $X_0$]\label{rem:pullback-dualizing}
Recall that for a nodal curve there is an explicit formula for the dualizing sheaf in terms of the normalization. In our setting, $X_0$ has a single node $p$, its normalization is $\pi\colon \widetilde X_0\to X_0$, and $\pi^{-1}(p) = \{x^+,x^-\}$. The \emph{dualizing sheaf $\omega_{X_0}$} is the kernel
\[
\omega_{X_0} \ = \ \ker\left[ \pi_*\Omega^1_{\widetilde X_0}(x^++x^-) \longrightarrow \mathbb C_p \right],
\]
where $\mathbb C_p$ denotes the skyscraper sheaf at $p$, and the map $\pi_*\Omega^1_{\widetilde X_0}(x^++x^-)\to \mathbb C_p$ is given by
\[
s \ \longmapsto \ \mathsf{Res}(s;x^+) + \mathsf{Res}(s;x^-),
\]
with $\mathsf{Res}(s;x)$ the residue of a form $s$ at a point $x$.

Pulling back along $\pi$ (the residue condition only constrains how the two branches at $p$ are glued together as a sheaf \emph{on $X_0$}; on the normalization the two branches are separated, so the constraint disappears) gives
\[
\pi^*\omega_{X_0} \ \cong \ \Omega^1_{\widetilde X_0}(x^++x^-).
\]
In particular, a Higgs field $\phi\colon E\to E\otimes\pi^*\omega_{X_0}$ as in the definition of a GPH is the same data as a map $\phi\colon E\to E\otimes\Omega^1_{\widetilde X_0}(x^++x^-)$, i.e.\ a Higgs field on $\widetilde X_0$ with at worst simple poles at $x^+,x^-$. This is what justifies speaking of the \emph{residues} $\phi_{x^+},\phi_{x^-}\in\End(E_{x^+}),\End(E_{x^-})$ throughout the paper (starting with $d^1$ in \cref{thm:def-GPH}), and identifies the hypothesis of \cref{TripDef} with the general definition rather than a separate setup.
\end{remark}

A generalized parabolic Hitchin pair $(E, F(E), \phi)$ induces a torsion-free sheaf \(\mathcal{F}\) on the nodal curve \(X_0\) as the kernel of the natural surjective map
\[
\mathcal{F} \ :=\ \ker\left( \pi_* E \ \longrightarrow\ \frac{E_{x^+} \oplus E_{x^-}}{F(E)} \right),
\]
and Higgs field \(\phi_{\mathcal{F}}: \mathcal F\to \mathcal F\otimes \omega_{X_0}\) because of the condition $(\pi_*\phi)(F(E)) \subseteq F(E) \otimes \omega_{X_0}$.

There is a natural morphism of moduli spaces (\cite[Theorem 2.9 ]{Bh2})
\[
\phi \colon \mathcal{M}_{\mathrm{GPH}} \longrightarrow \mathcal{M}_{\mathrm{TFH}},
\]
from the moduli space of semistable generalized parabolic Hitchin pairs to the moduli space of semistable torsion-free Higgs sheaves of rank \(n\) and degree \(d\) on \(X_0\). 

\

Let $\mathbb C[\epsilon]$ be the ring of dual numbers, i.e., $\epsilon^2=0$. We denote $\mathbb D:=Spec~~\mathbb C[\epsilon]$. Consider the nodal curve $X_0\times \mathbb D$ with the nodal divisor $p\times \mathbb D$. We the following map $\pi_{\mathbb D}: \widetilde{X_0}\times \mathbb D\rightarrow X_0\times \mathbb D$. Then $\pi_{\mathbb D}^{-1}({p\times \mathbb D})=(x^+\times \mathbb D) ~\coprod~ (x^-\times \mathbb D)$.

\

\begin{definition}
Let \((E, \phi, F(E))\) be a generalized parabolic Hitchin pair of rank $n$ on the nodal curve \(X_0\). 
A \textbf{first-order infinitesimal deformation} of \((E, \phi, F(E))\) consists of a triple 
\((E_{\mathbb{D}}, \phi_{\mathbb{D}}, F(E_{\mathbb{D}}))\) over \(X_0 \times \operatorname{Spec} \mathbb{C}[\epsilon]\) 
(with \(\epsilon^2 = 0\)), where:
\begin{enumerate}
\item \((E_{\mathbb{D}}, \phi_{\mathbb{D}})\) is a first-order infinitesimal deformation of the underlying Hitchin pair \((E, \phi)\);

\item \(F(E_{\mathbb{D}}) \subset E_{\mathbb{D},x^+} \oplus E_{\mathbb{D},x^-}\) is a \(\mathbb{C}[\epsilon]\)-submodule such that
  \begin{enumerate}
  \item the quotient \(\dfrac{E_{\mathbb{D},x^+} \oplus E_{\mathbb{D},x^-}}{F(E_{\mathbb{D}})}\) is a locally free \(\mathbb{C}[\epsilon]\)-module of rank \(n\);
  \item \(F(E_{\mathbb{D}}) \otimes_{\mathbb{C}[\epsilon]} \mathbb{C} \ \cong\ F(E)\) (i.e., it reduces to the original subspace modulo \(\epsilon\));
  \end{enumerate}

\item the deformed Higgs field preserves the deformed generalized parabolic structure:
  \[
  (\pi_{\mathbb{D}})_* \phi_{\mathbb{D}} \bigl( F(E_{\mathbb{D}}) \bigr) 
  \ \subseteq\ 
  F(E_{\mathbb{D}}) \otimes \omega_{X_0 \times \mathbb D/ \mathbb D}.
  \]
\end{enumerate}
We denote such a deformation by \((E_{\mathbb{D}}, \phi_{\mathbb{D}}, F(E_{\mathbb{D}}))\).
\end{definition}
\

\begin{definition}
Two such infinitesimal deformations $(E_{\mathbb D}, \phi_{\mathbb D}, F(E_{\mathbb D}))$ and $(E'_{\mathbb D}, \phi'_{\mathbb D}, F(E'_{\mathbb D}))$ are said to be \textbf{isomorphic} if there exists an isomorphism of vector bundles $\psi_{\mathbb D}\colon E_{\mathbb D}\to E'_{\mathbb D}$ such that
\begin{enumerate}
\item $\psi_{\mathbb D}|_{\operatorname{Spec}\mathbb C} = \mathrm{Identity}$;
\item $\psi_{\mathbb D}$ intertwines the Higgs fields:
\[
{\pi_{\mathbb D}}_*(\psi_{\mathbb D}\times \mathbb I_{\mathbb D})\circ {\pi_{\mathbb D}}_*(\phi_{\mathbb D}) \ = \ {\pi_{\mathbb D}}_*(\phi'_{\mathbb D})\circ {\pi_{\mathbb D}}_*(\psi_{\mathbb D});
\]
\item $\psi_{\mathbb D}$ carries the flag onto the flag: writing $\psi_{\mathbb D,D}\colon E_{\mathbb D,D}\to E'_{\mathbb D,D}$ for the isomorphism induced by $\psi_{\mathbb D}$ on the fibres over the node,
\[
\psi_{\mathbb D,D}\bigl(F(E_{\mathbb D})\bigr) \ = \ F(E'_{\mathbb D}).
\]
\end{enumerate}
\end{definition}

\

\begin{theorem}\label{thm:def-GPH}
The isomorphism classes of first-order infinitesimal deformations of a GPH \((E_{\bullet},\phi):=(E,\phi, F(E))\) 
are canonically parametrized by the first hypercohomology \(\mathbb{H}^1(\mathcal{C}_{(E_{\bullet}, \phi)})\), 
where \(\mathcal{C}_{(E_{\bullet}, \phi)}\) is the following complex of sheaves on the nodal curve \(X_0\):
\[
\mathcal{C}_{(E_{\bullet}, \phi)} \colon \ 
\pi_*\End(E) \ \xrightarrow{d^0}\ 
\pi_*\End(E)\otimes\omega_{X_0} \oplus \Hom(F,E_D/F) 
\ \xrightarrow{d^1}\ 
\Hom(F,E_D/F).
\]
The differentials are given explicitly as follows:
\begin{enumerate}
\item The first differential \(d^0\) is
\[
d^0(\psi) = \bigl( [\psi, \phi],\ \overline{\psi \circ F} \bigr),
\]
where \([\psi, \phi] = \psi\phi - \phi\psi\) is the commutator, and \(\overline{\psi \circ F}\) denotes the composition \(\psi \circ F\) followed by the projection onto the quotient \(E_D/F(E)\).
\item The second differential \(d^1\) is
\[
d^1(s, G) = \overline{s \circ F} \;+\; \overline{(\pi_*\phi)_p\circ \tilde G} \;-\; G\circ\Bigl((\pi_*\phi)_p\big|_{F(E)}\Bigr),
\]
where \(s \in \pi_*\End(E)\otimes \omega_{X_0}\), \(G \in \Hom(F(E), E_D/F(E))\), the map \((\pi_*\phi)_p\big|_{F(E)}\colon F(E)\to F(E)\) is the restriction of \((\pi_*\phi)_p\) to \(F(E)\) (well defined since the Higgs field \(\phi\) is compatible with the flag \(F(E)\)), and \(\overline{(\pi_*\phi)_p\circ \tilde G}\) is the composition
\[
F(E)\xrightarrow{\tilde{G}} E_{x^+}\oplus E_{x^-}\xrightarrow{\;\mathrm{res}\,\phi_{x^+}\,\oplus\,\mathrm{res}\,\phi_{x^-}\;} E_{x^+}\oplus E_{x^-}\to \frac{E_{x^+}\oplus E_{x^-}}{F(E)},
\]
where \(\tilde{G}\colon F(E)\to E_{x^+}\oplus E_{x^-}\) is any choice of lift of \(G\); two choices of lift give the same map because \(\phi\) is compatible with the flag \(F(E)\). The last term \(G\circ\bigl((\pi_*\phi)_p|_{F(E)}\bigr)\) requires no lift, since \((\pi_*\phi)_p|_{F(E)}\) already lands in \(F(E)\), the domain on which \(G\) is defined.
\end{enumerate}
\end{theorem}

\begin{remark}[Non-surjectivity of the flag component of $d^0$]\label{rem:flag-not-surjective}
It is natural to ask whether the deformation theory of \cref{thm:def-GPH} could be simplified by working with a single ``kernel sheaf'' $\ker(d^0)\subset \pi_*\End(E)$ in place of the full complex $\mathcal{C}_{(E_{\bullet},\phi)}$. In the literature this sheaf $\ker(d^0)$ is often denoted as $\mathcal Par\End(E)$. This shortcut is only legitimate when the short exact sequence
\[
0 \to \ker(d^0) \to \mathcal{C}^0 \to \mathcal{C}^1 \to 0
\]
holds, i.e.\ when $d^0$ is surjective as a map of sheaves. We show below that this fails already for the flag component
\[
\pi_*\End(E) \ \longrightarrow\ \Hom(F,E_D/F), \qquad \psi \longmapsto \overline{\psi \circ F},
\]
Since $\Hom(F,E_D/F)$ is a skyscraper sheaf at the node, this map is determined entirely by the fibre map
\[
L \colon \End(E_{x^+}) \oplus \End(E_{x^-}) \longrightarrow \Hom(F,E_D/F), \qquad
L(\psi_+,\psi_-) = \overline{(\psi_+ \oplus \psi_-)\circ F}.
\]
Consider the simplest case when the flag is given as the graph of a linear map $A: E_{x^+}\to E_{x^-}$ i.e., $F = \Gamma_A = \{(v,Av) : v \in E_{x^+}\}$ and identifying $F \cong E_{x^+}$, $E_D/F \cong E_{x^-}$ via $(x,y) \mapsto y - Ax$, one computes
\[
L(\psi_+,\psi_-) \ = \ \psi_- A - A \psi_+ \ \in \ \Hom(E_{x^+},E_{x^-}).
\]

\emph{Explicit counterexample.} Take $n=2$, $E_{x^+} = E_{x^-} = \mathbb{C}^2$ with basis $e_1,e_2$, and
\[
A = \begin{pmatrix} 1 & 0 \\ 0 & 0 \end{pmatrix}, \qquad \rank A = 1.
\]
For any $\psi_+,\psi_- \in M_2(\mathbb{C})$ the $(2,2)$-entry of $\psi_- A - A \psi_+$ vanishes identically. Hence
\[
E_{22} = \begin{pmatrix} 0 & 0 \\ 0 & 1 \end{pmatrix} \ \notin \ \operatorname{im}(L),
\]
so $L$, and therefore $\pi_*\End(E) \to \Hom(F,E_D/F)$, is \emph{not} surjective; its cokernel is $1$-dimensional, spanned by $E_{22}$. 

Consequently $d^0$ is not surjective in general. This is why the deformation and obstruction theory of a GPH must be computed via the hypercohomology of the full three-term complex $\mathcal{C}_{(E_{\bullet},\phi)}$, rather than via the cohomology of any single subsheaf like $\mathcal Par\End(E)$.
\end{remark}

\begin{remark}
Before discussing the proof let us verify that \(d^1 \circ d^0 = 0\). Let \(\psi \in C^0(\pi_*\End(E))\). Set
\[
s = [\psi, \phi] = \psi\phi - \phi\psi, \qquad G = \overline{\psi \circ F}.
\]
Since \(\phi\) is compatible with the flag \(F(E)\), the map \(\phi \circ F\) factors through \(F\):
\[
\phi \circ F = F \circ \Bigl((\pi_*\phi)_p\big|_{F(E)}\Bigr).
\]
Then
\[
s \circ F = \psi\circ(\phi \circ F) - \phi\circ(\psi \circ F) = (\psi \circ F)\circ\Bigl((\pi_*\phi)_p\big|_{F(E)}\Bigr) - \phi\circ(\psi \circ F),
\]
so, projecting to \(E_D/F(E)\),
\[
\overline{s \circ F} = \overline{\psi \circ F}\circ\Bigl((\pi_*\phi)_p\big|_{F(E)}\Bigr) - \overline{\phi\circ(\psi \circ F)} = G\circ\Bigl((\pi_*\phi)_p\big|_{F(E)}\Bigr) - \overline{(\pi_*\phi)_p\circ \tilde G},
\]
where the last equality uses the fact that \(\psi \circ F\) is itself a lift of \(G = \overline{\psi \circ F}\), so that \(\overline{\phi\circ(\psi\circ F)} = \overline{(\pi_*\phi)_p\circ \tilde G}\) by definition of the latter.

Therefore
\[
d^1(s,G) = \overline{s \circ F} + \overline{(\pi_*\phi)_p\circ \tilde G} - G\circ\Bigl((\pi_*\phi)_p\big|_{F(E)}\Bigr) = 0,
\]
since the two occurrences of \(G\circ\bigl((\pi_*\phi)_p|_{F(E)}\bigr)\) and \(\overline{(\pi_*\phi)_p\circ \tilde G}\) cancel.
\end{remark}

\begin{proof}
Let \(((E_{\mathbb{D}}, \phi_{\mathbb{D}}), F(E_{\mathbb{D}}))\) be a first-order infinitesimal deformation of \((E, \phi, F(E))\) over \(\operatorname{Spec} \mathbb{C}[\epsilon]/(\epsilon^2)\).

Choose an open affine cover \(\{U_i\}_{i\in I}\) of \(X_0\) such that the preimages \(V_i = \pi^{-1}(U_i)\) trivialize \(E\). Over each \(V_i \times \operatorname{Spec}\mathbb{C}[\epsilon]\) we have
\[
E_{\mathbb{D}}|_{V_i} \ \simeq\ E|_{V_i} \otimes_{\mathbb{C}} \mathbb{C}[\epsilon],
\]
with transition functions
\[
g_{ij} = I + \epsilon \cdot B_{ij}, \qquad B_{ij} \in \Gamma(V_i \cap V_j, \End(E)).
\]

The deformed Higgs field is given locally by
\[
\phi_{\mathbb{D}}|_{V_i} = \phi + \epsilon \cdot s_i, \qquad s_i \in \Gamma(V_i, \End(E) \otimes \Omega^1_{\widetilde{X_0}}(x^++x^-)).
\]

The deformed flag is given locally by
\[
F(E_{\mathbb{D}})|_{V_i} \ = \ \bigl\{ F(x) + \epsilon \cdot \tilde G_i(x) \ \big|\ x \in \mathbb C^n \bigr\}, \footnote{Here $n$ is the rank of $E$, and $F(E)$ is seen as a linear map $F: \mathbb C^n\to E_D$}
\]
where \(\tilde G_i \in \Hom(F(E), E_{D})|_{U_i}\) \footnote{Notice that $\Hom(F(E), E_{D})$ is not in general a sheaf on $\widetilde{X_0}$ but always a sheaf on $X_0$.}.

\smallskip

The data \((B_{ij}, s_i, \tilde G_i)\) must satisfy three compatibility conditions, which we now derive:

\begin{itemize}
\item \textbf{Hitchin cocycle condition} (from deformation of the underlying bundle and the Higgs field): \\
 The deformed Higgs field \(\phi_{\mathbb{D}}\) must be compatible with the deformed transition functions \(g_{ij}\). Expanding to first order gives
  \begin{equation}
  s_j - s_i = [B_{ij}, \phi] \quad \text{on } V_i \cap V_j.
  \end{equation}

\item \textbf{Parabolic transition condition} (from deformation of the flag): \\
 The deformed flags on the overlapping charts must agree after applying the deformed transition functions. Expanding
  \[
  (I + \epsilon B_{ij}) \cdot (F(x) + \epsilon \tilde G_i(x)) = F(x) + \epsilon (B_{ij} \circ F(x) + \tilde G_i(x))
  \]
  and comparing with the expression on \(V_j\) yields
  \begin{equation}
  \overline{B_{ij} \circ F} + G_i - G_j = 0 \quad \text{in } \Hom(F(E), E_D/F(E))|_{U_i\cap U_j}. \footnote{Here $G_i=\tilde G_i$ mod $Image~F(E)$ and $G_j=\tilde G_j$ mod $Image~F(E)$ and $\overline{B_{ij} \circ F}:={B_{ij} \circ F}$ mod $Image~F(E)$.}
  \end{equation}

\item \textbf{Higgs--flag compatibility condition}:\\
Let \(F \colon F \to E_D\) denote the inclusion of the flag (as elsewhere, we use \(F\) for both the flag and the underlying vector space). Since \(F\) is \(\phi\)-invariant, \(\phi \circ F\) factors through \(F\): there is a (unique) map \(\phi \colon F \to F\) with
\[
\phi \circ F \;=\; F \circ \phi,
\]
i.e.\ the square
\[
\begin{array}{ccc}
F & \xrightarrow{\ \ Id\ \ } & E_D \\[4pt]
{\scriptstyle \phi}\big\downarrow & & \big\downarrow{\scriptstyle \phi} \\[4pt]
F & \xrightarrow{\ \ Id\ \ } & E_D
\end{array}
\]
commutes (we write \(\phi\) for both the restriction and the ambient operator, the domain making clear which is meant).

The deformed object \((E_{\mathbb{D}}, \phi_{\mathbb{D}}, F(E_{\mathbb{D}}))\) must still be a generalized parabolic Higgs bundle to first order. Let \(\tilde G_i \colon F \to E_D\) be a map, and set
\[
F(E_{\mathbb{D}}) := \operatorname{Im}\bigl(F+\epsilon \tilde G_i\bigr) \subset E_{D,\mathbb D}, \qquad \phi_{\mathbb{D}} := \phi + \epsilon\, s_i,
\]
where \(s_i \in \Gamma(V_i, \End(E) \otimes \Omega^1_{\widetilde{X_0}}(x^++x^-))\), so that \(F+\epsilon \tilde G_i \colon F \to E_{D,\mathbb{D}}\) is the deformed inclusion. Since \(F(E_{\mathbb D})\) is defined as an image, it is unchanged if \(\tilde G_i\) is altered by a map landing in \(\operatorname{Im}(F)\); only its composite with the quotient \(E_D \to E_D/\operatorname{Im}(F)\) matters, and it is in this sense that the compatibility condition below should be read.

The condition \(\phi_{\mathbb{D}}\bigl(F(E_{\mathbb{D}})\bigr) \subseteq F(E_{\mathbb{D}})\), expanded to first order, says precisely that
\[
(\phi+\epsilon s_i)\circ(F+\epsilon \tilde G_i) \;=\; (F+\epsilon \tilde G_i)\circ\phi \pmod{\epsilon^2},
\]
i.e.\ that the square
\[
\begin{array}{ccc}
F & \xrightarrow{\ \ F+\epsilon \tilde G_i\ \ } & E_{D,\mathbb D} \\[4pt]
{\scriptstyle \phi}\big\downarrow & & \big\downarrow{\scriptstyle \phi+\epsilon s_i} \\[4pt]
F & \xrightarrow{\ \ F+\epsilon \tilde G_i\ \ } & E_{D,\mathbb D}
\end{array}
\]
commutes.

Expanding both sides for \(v \in F\) modulo \(\epsilon^2\), and using \(\phi\circ F = F\circ\phi\):
\[
(\phi+\epsilon s_i)\bigl(F(v)+\epsilon \tilde G_i(v)\bigr) = F(\phi(v)) + \epsilon\bigl[\phi\circ \tilde G_i(v) + s_i\circ F(v)\bigr],
\]
\[
\bigl(F+\epsilon \tilde G_i\bigr)\bigl(\phi(v)\bigr) = F(\phi(v)) + \epsilon\, \tilde G_i\bigl(\phi(v)\bigr).
\]
The $0$-th $\epsilon$-order terms agree automatically. Equating the $1$-th $\epsilon$ order terms modulo \(\operatorname{Im}(F)\) gives, for all \(v \in F\),
\[
\overline{\phi\circ \tilde G_i}(v) + \overline{s_i\circ F}(v) = \overline{(\tilde G_i\circ \phi)}(v),
\]
that is,
\begin{equation}
\overline{s_i\circ F} \;+\; \overline{\phi\circ \tilde G_i} \;-\; \overline{\tilde G_i\circ \phi} \;=\; 0.
\end{equation}

\end{itemize}

These three conditions are precisely the defining cocycle relations for the complex \(\mathcal{C}_{(E_{\bullet}, \phi)}\). For the reader's convenience we write here the details of the total complex of the Čech double complex:

The total complex of the Čech double complex for a GPH is
\[
0 \longrightarrow \Tot^0 \xrightarrow{d^0_{\Tot}}
\Tot^1 \xrightarrow{d^1_{\Tot}}
\Tot^2  \longrightarrow 0,
\]
where
\[
\begin{aligned}
\Tot^0 &= C^0(\pi_*\End(E)), \\
\Tot^1 &= C^1(\pi_*\End(E)) \ \oplus\ C^0(\pi_*(\End(E)\otimes\omega_{X_0})) 
\ \oplus\ C^0(\Hom(F,E_D/F)), \\
\Tot^2 &= C^2(\pi_*\End(E)) \ \oplus C^1(\pi_*(\End(E)\otimes\omega_{X_0})) \ \oplus\ C^1(\Hom(F,E_D/F)) \\
&\qquad 
\ \oplus\ C^0(\Hom(F,E_D/F)).
\end{aligned}
\]
The differentials are defined explicitly as follows:
\begin{align*}
d^0_{\Tot}(\{\psi_i\})
&= \bigl( \psi_j - \psi_i,\ 
[\psi_i, \phi],\ 
\overline{\psi_i \circ F}:=\psi_i \circ F \pmod{im F} \bigr), \\[6pt]
d^1_{\Tot}(B_{ij},\ s_i,\ G_i)
&= \Bigl(
\underbrace{B_{jk} - B_{ik} + B_{ij}}_{\text{Čech on bundle}}, \ 
\underbrace{s_j - s_i - [B_{ij}, \phi]}_{\text{Čech + commutator on Higgs}}, \\
&\qquad 
\underbrace{G_j - G_i - \overline{B_{ij} \circ F} }_{\text{flag compatibility}}, \ 
\underbrace{\overline{s_i\circ F} + \overline{(\pi_*\phi)_p\circ G_i} - G_i\circ\Bigl((\pi_*\phi)_p\big|_{F(E)}\Bigr)}_{\text{Higgs--flag compatibility}} 
\Bigr). 
\end{align*} \footnote{It is worth mentioning that here by $\phi\circ \tilde G_i$ we really mean $(\pi_*\phi)_p\circ G_i$. So in that sense, $\phi\circ \tilde G_i=(\pi_*\phi)_p\circ \tilde G_i$.}

\smallskip

The collection \(\{(B_{ij}, s_i, G_i)\}\) therefore forms a Čech 1-cocycle for \(\mathcal{C}_{(E_{\bullet}, \phi)}\) with respect to the cover \(\{U_i\}\). Changing the choice of trivializations alters the cocycle by a Čech coboundary. Hence every first-order deformation determines a well-defined class in \(\mathbb{H}^1(\mathcal{C}_{(E_{\bullet}, \phi)})\).

Conversely, given any 1-cocycle \((B_{ij}, s_i, G_i)\) of \(\mathcal{C}_{(E_{\bullet}, \phi)}\), one can glue the local pieces
\[
(E|_{V_i} \otimes \mathbb{C}[\epsilon],\ \phi + \epsilon s_i,\ F + \epsilon G_i)
\]
using the above cocycle conditions to obtain a global first-order deformation of the original GPH.

Finally, suppose
\[
\psi_{\mathbb{D}} = I + \epsilon \psi \colon ((E_{\mathbb{D}}, \phi_{\mathbb{D}}), F(E_{\mathbb{D}})) 
\xrightarrow{\ \sim\ } ((E'_{\mathbb{D}}, \phi'_{\mathbb{D}}), F(E'_{\mathbb{D}}))
\]
is an isomorphism of deformations. A direct computation shows that the difference of the corresponding cocycles is a Čech coboundary:
\begin{enumerate}
\item \(B'_{ij} - B_{ij} = \psi_j - \psi_i\),
\item \(s'_i - s_i = [\psi_i, \phi]\),
\item \(G'_i - G_i = \psi_i \circ F \pmod{\operatorname{im} F}\).
\end{enumerate}
Conversely, if two cocycles differ by a coboundary given by local sections \(\{\psi_i\}\), one can construct an isomorphism using exactly these sections.

Thus there is a canonical bijection between isomorphism classes of first-order infinitesimal deformations of \((E_{\bullet},\phi)\) and \(\mathbb{H}^1(\mathcal{C}_{(E_{\bullet}, \phi)})\).

\end{proof}

Let us also consider the special case where the generalized parabolic structure i.e., the subspace $  F(E) \subset E_{x^+} \oplus E_{x^-}  $ is the graph of an isomorphism $  A \colon E_{x^+} \to E_{x^-} $. In this situation, it is often more convenient to work directly with the gluing datum $  A  $ rather than the subspace $  F(E)  $. The following theorem describes the deformation complex for such a Higgs triple $  (E, A, \phi)  $, which encodes both the Higgs field and the nodal gluing.

\begin{theorem}\label{TripDef}
Let \((E, A, \phi)\) be a Higgs triple on the normalization \(\widetilde{X}_0\) of a nodal curve \(X_0\), where:
\begin{itemize}
\item \(E\) is a vector bundle on \(\widetilde{X}_0\),
\item \(A: E_{x^+} \to E_{x^-}\) is an isomorphism,
\item \(\phi: E \to E \otimes \Omega^1_{\widetilde{X}_0}(x^+ + x^-)\) is a Higgs field that satisfies the compatibility condition
  \[
  A \circ \phi_{x^+} = -\phi_{x^-} \circ A,
  \]
  where \(\phi_{x^+}\) and \(\phi_{x^-}\) denote the residues of \(\phi\) at those points and the minus sign comes from the identification of the Poincaré residue (multiplication by \(-1\)).
\end{itemize}
The isomorphism classes of first-order infinitesimal deformations of the triple \((E, A, \phi)\) are canonically parametrized by the first hypercohomology \(\mathbb{H}^1(\mathcal{C}_{(E,A,\phi)})\), where the deformation complex is
\[
\mathcal{C}_{(E,A,\phi)} \colon \
\pi_*\End(E) \ \xrightarrow{d^0}\
\pi_*(\End(E))\otimes \omega_{X_0} \oplus \Hom(E_{x^+}, E_{x^-})
\ \xrightarrow{d^1}\
\Hom(E_{x^+}, E_{x^-}).
\]
The differentials are given explicitly by
\[
d^0(\psi) = \bigl( [\psi, \phi],\ \psi_{x^-} \circ A + A \circ \psi_{x^+} \bigr),
\]
\[
d^1(s, G) = ({s_{x^-} \circ A + A \circ s_{x^+}}) - (\phi_{x^-} \circ G + G \circ \phi_{x^+}).
\]
\end{theorem}

\begin{proof}
The proof is similar to the proof of the Theorem \ref{thm:def-GPH}.
\end{proof}

\section{\textbf{Computations of the Hypercohomologies}}\label{sec:computations}
In this section, we apply the deformation complex \(\mathcal{C}_{(E_{\bullet},\phi)}\) introduced in Section $3$ to compute the hypercohomology groups \(\mathbb{H}^\bullet(\mathcal{C}_{(E_{\bullet},\phi)})\). These groups control infinitesimal deformations (\(\mathbb{H}^1\)) and obstructions (\(\mathbb{H}^2\)), and play a decisive role in the proof of normality of the moduli spaces in Section 5.

Two main computational tools are used repeatedly: a short exact sequence of complexes and the hypercohomology spectral sequence. For the reader's convenience, both tools are developed in detail in Appendix~\ref{Append}. Here we recall the key conclusions and apply them to obtain explicit results, most notably the one-dimensionality of the obstruction space \(\mathbb{H}^2(\mathcal{C}_{(E_{\bullet},\phi)})\) when the generalized parabolic structure has rank \(n-1\).

\begin{lemma}\label{H2:Vect}
Let \((E_{\bullet}, \phi)\) be a generalized parabolic Higgs bundle on \(\widetilde{X_0}\) with \(\operatorname{rank} A = n\) (i.e., the gluing map \(A \colon E_{x^+} \to E_{x^-}\) is an isomorphism) and $E$ is stable. Then
\[
\dim \mathbb{H}^2(\mathcal{C}_{(E_{\bullet},\phi)}) = 1.
\]
\end{lemma}

\begin{proof}
By the hypercohomology spectral sequence (see Appendix \ref{Append}), we have
\[
\mathbb{H}^2(\mathcal{C}^\bullet) \ \simeq \ \operatorname{coker}\Bigl( H^0(\mathcal{C}^1) \to H^0(\mathcal{C}^2) \Bigr),
\]
since the term \(E_2^{1,1}\) vanishes by stability of the underlying vector bundle \(E\).

When \(\operatorname{rank} A = n\), using Remark \ref{Rem2}, one can assume $A=I_n$, the identity matrix. The differential \(d^1 \colon \mathcal{C}^1 \to \mathcal{C}^2\) acts on global sections by
\[
d^1(\psi, f) = (\psi_{x^-} + \psi_{x^+}) - (-\phi_{x^+} \circ f + f \circ \phi_{x^+} ),
\]
where \(\psi \in H^0(\widetilde{X_0}, \End(E) \otimes \Omega^1_{\widetilde{X_0}}(x^++x^-))\) and \(f \in \Hom(E_{x^+}, E_{x^-})\).

Consider the short exact sequence on the normalization \(\widetilde{X_0}\),
\[
0 \to \End(E)\otimes \Omega^1_{\widetilde{X_0}} \to \End(E)\otimes \Omega^1_{\widetilde{X_0}}(x^++x^-) 
\to \End(E_{x^+}) \oplus \End(E_{x^-}) \to 0.
\]
Stability of \(E\) implies that the residue map
\[
H^0\bigl(\widetilde{X_0}, \End(E)\otimes \Omega^1_{\widetilde{X_0}}(x^++x^-)\bigr) 
\twoheadrightarrow 
\bigl\{ (X_+, X_-) \in \End(E_{x^+}) \oplus \End(E_{x^-}) 
\bigm| \operatorname{tr}(X_+) + \operatorname{tr}(X_-) = 0 \bigr\}
\]
is surjective.

Let \(X \in \Hom(E_{x^+}, E_{x^-})\). The equation
\[
(\psi_{x^-} + \psi_{x^+}) - (-\phi_{x^+} \circ f + f \circ \phi_{x^+}) = X
\]
has a solution \((\psi, f)\) if and only if \(\operatorname{tr}(X) = 0\). Indeed:
\begin{itemize}
\item The left-hand side always satisfies \(\operatorname{tr}(\text{LHS}) = 0\), since \(\operatorname{tr}(-\phi_{x^+} \circ f + f \circ \phi_{x^+}) = 0\) (trace of a commutator) and \(\operatorname{tr}(\psi_{x^-} + \psi_{x^+}) = 0\).
\item For any \(f\) and any trace-zero \(X\), one can choose residues \(\psi_{x^+}\) and \(\psi_{x^-}\) satisfying the trace-zero condition, which lift to a global section \(\psi\) by the surjectivity of the residue map.
\end{itemize}

Therefore, the cokernel is one-dimensional, and
\[
\dim \mathbb{H}^2(\mathcal{C}_{(E_{\bullet},\phi)}) = 1.
\]
\end{proof}

\begin{remark}\label{Rem2}

We are allowed to choose any convenient normal form for the linear map \(A: E_{x^+} \to E_{x^-}\) when performing explicit computations at the node. 

More precisely, given a generalized parabolic structure \(F(E) = \operatorname{graph}(A)\), we may simultaneously change bases in the fibers \(E_{x^+}\) and \(E_{x^-}\). This corresponds to replacing 

$$A \ \longmapsto \ Q^{-1} A P$$

for invertible matrices \(P \in \mathrm{GL}(E_{x^+})\) and \(Q \in \mathrm{GL}(E_{x^-})\). This change does not alter the isomorphism class of the GPH \((E, \phi, F(E))\) and leaves all geometric invariants (including the deformation complex \(\mathcal{C}_{(E_\bullet, \phi)}\) up to the isomorphism) unchanged.

For any linear map \(A \colon E_{x^+} \to E_{x^-}\) with \(\operatorname{rank} A = n-1\), we may always choose bases of \(E_{x^+}\) and \(E_{x^-}\) such that

$$
A = \begin{pmatrix}

I_{n-1} & 0 \\

0 & 0

\end{pmatrix}
$$

Indeed, \(\dim\ker A=1\) and \(\dim\operatorname{im} A=n-1\); pick a basis of \(E_{x^+}\) with last vector spanning the kernel and the first \(n-1\) vectors mapping to a basis of the image, then extend the latter to a basis of \(E_{x^-}\). All such maps lie in a single orbit under the natural action \(\mathrm{GL}(E_{x^+})\times\mathrm{GL}(E_{x^-})\) given by \(A\mapsto Q^{-1}AP\).

\end{remark}

\begin{lemma}\label{hypecomp}
Let \((E, \phi, A \colon E_{x^+} \to E_{x^-})\) be a generalized parabolic Higgs bundle on \(\widetilde{X_0}\) with \(\operatorname{rank} A = n-1\) such that the underlying vector bundle \(E\) is stable. Then
\[
\dim \mathbb{H}^2(\mathcal{C}_{(E_{\bullet},\phi)}) = 1.
\]
\end{lemma}

\begin{proof}
From the hypercohomology spectral sequence (Appendix \ref{Append}), we have
\[
\mathbb{H}^2(\mathcal{C}^\bullet) \ \simeq \ \operatorname{coker}\bigl( H^0(\mathcal{C}^1) \to H^0(\mathcal{C}^2) \bigr).
\]

We work in the normal form \(A = \operatorname{diag}(I_{n-1}, 0)\). Fix a basis \(\{e_1, \dots, e_{n-1}, e_n\}\) of \(E_{x^+}\) (and correspondingly at \(x^-\)) such that \(\ker A = \langle e_n \rangle\). In this basis, \(\dim H^0(\mathcal{C}^2) = n^2\).

\textbf{Local description of the differential \(d^1\) at the node:} 
Let \(\psi\) have residues at \(x^+\) and \(x^-\) respectively given by
\[
\psi_+ = \begin{pmatrix} B & u \\ v & a \end{pmatrix}, \qquad
\psi_- = \begin{pmatrix} C & w \\ z & b \end{pmatrix},
\]
with the trace condition \(\operatorname{tr}(\psi_+) + \operatorname{tr}(\psi_-) = 0\).

The differential \(d^1\) at the node is
\[
d^1(\psi, f) = \psi_{x^-} \circ A + A \circ \psi_{x^+} - (\phi_{x^-} \circ f + f \circ \phi_{x^+}).
\]

Since \(A = \operatorname{diag}(I_{n-1}, 0)\), we have
\[
\psi_{x^-} \circ A + A \circ \psi_{x^+} = \begin{pmatrix} B + C & u \\ z & 0 \end{pmatrix}.
\]

From the compatibility condition \(-A\phi_{x^+} = \phi_{x^-}A\) and the trace condition \(\operatorname{tr}(\phi_{x^+}) + \operatorname{tr}(\phi_{x^-}) = 0\), the residues of the Higgs field take the form
\[
\phi_{x^+} = \begin{pmatrix} M & 0 \\ n & l \end{pmatrix}, \qquad
\phi_{x^-} = \begin{pmatrix} -M & m \\ 0 & -l \end{pmatrix}.
\]

Let \(f\) be written in block form as
\[
f = \begin{pmatrix} P & Q \\ R & S \end{pmatrix}.
\]
Then
\[
\phi_{x^-} \circ f + f \circ \phi_{x^+} =
\begin{pmatrix}
m R + Q n + [P, M] & m S - M Q + Q l \\
R(M - l) + S n & 0
\end{pmatrix},
\]
where \([P,M] = PM - MP\).

Combining both parts, the full expression for the differential is
\[
d^1(\psi, f) =
\begin{pmatrix}
B + C - (m R + Q n + [P,M]) & u - (m S - M Q + Q l) \\
z - \bigl(R(M-l) + S n\bigr) & 0
\end{pmatrix}.
\]

We now show that the image of \(d^1\) has codimension 1 inside \(H^0(\mathcal{C}^2) \simeq M_n(\mathbb{C})\).

Consider an arbitrary matrix of the form
\[
X = \begin{pmatrix}
x & y \\
t & 0
\end{pmatrix}.
\]
To solve \(d^1(\psi, f) = X\), we need to solve the system
\[
x = B + C - (m R + Q n + [P,M]), \quad
y = u - (m S - M Q + Q l),
\]
\[
t = z - \bigl(R(M-l) + S n\bigr).
\]

This system can be solved by taking \(P = Q = R = S = 0\), choosing \(C\) arbitrarily, setting \(x := B + C\), \(y := u\), \(t := z\), and finally choosing \(a := -b - \operatorname{tr}(x)\) to satisfy the trace condition. 

Thus every matrix with zero in the \((n,n)\)-position lies in the image of \(d^1\). However, the \((n,n)\)-entry of \(d^1(\psi,f)\) is always zero. Therefore the cokernel is one-dimensional, generated by the matrix \(E_{nn}\) (the matrix with 1 in the \((n,n)\)-entry and zeros elsewhere).

Hence
\[
\dim \mathbb{H}^2(\mathcal{C}_{(E_{\bullet},\phi)}) = \dim \operatorname{coker}\bigl( H^0(\mathcal{C}^1) \to H^0(\mathcal{C}^2) \bigr) = 1,
\]
as required.
\end{proof}

\begin{remark}
Although detailed definitions and construction of the deformation complex in Section~2 are written for an irreducible nodal curve with a single node, the results of Theorem~\ref{NT} (and the underlying deformation theory) extend straightforwardly to any nodal curve with an arbitrary number of nodes. The only modifications required are local at each node; the global arguments remain unchanged.
\end{remark}

\section{\textbf{Irreducibility of the moduli of GPHs}}

In this section we prove that the moduli stack of generalized parabolic Hitchin pairs \(\mathrm{GPH}\) is irreducible. The main tool is the comparison with the moduli stack of Gieseker--Higgs bundles \(\mathrm{GHB}\), which is known to be irreducible by \cite[Theorem A.12]{BDP}. We first establish a surjective morphism \(\mathrm{GHB} \to \mathrm{TFH}\), and then using similar ideas we show that the natural projection map $\widetilde{\mathrm{GHB}}\to \mathrm{GPH}$ is also surjective. This implies irreducibility of \(\mathrm{GPH}\).

We begin by recalling the relevant functors.

\begin{definition}\cite{BBN}\label{mod11}
Let \(X_0\) be the irreducible nodal curve. A \emph{semistable modification} (or expanded degeneration) of \(X_0\) is a flat proper morphism \(\pi \colon \mathfrak X \to X_0\) such that \(\mathfrak X\) is a semistable curve (i.e., a nodal curve with rational tree components attached at the node) and \(\pi\) is an isomorphism away from the node.

The functor \(\mathrm{GHB}\) (Gieseker--Higgs Bundles) from the category of schemes to groupoids is defined as follows: to a scheme \(S\) it associates the groupoid \(\mathrm{GHB}(S)\) whose objects are pairs \((X \to X_0 \times S, \mathcal{E}, \Phi)\) consisting of
\begin{itemize}
\item a flat family of semistable modifications \(\pi \colon \mathfrak X \to X_0 \times S\) (where each fibre is a semistable curve obtained by inserting a chain of exceptional rational components at the node),
\item a flat family \(\mathcal{E}\) of Gieseker bundles on \(\mathfrak X\) (in the sense of Gieseker--Seshadri--Nagaraj), of rank \(n\) and degree \(d\),
\item a Higgs field \(\Phi \colon \mathcal{E} \to \mathcal{E} \otimes_{\mathcal{O}_X} \omega_{\mathfrak XS}\),
\end{itemize}
up to isomorphism. Morphisms in the groupoid are isomorphisms of such data preserving the contraction to \(X_0 \times S\).
\end{definition}

\begin{definition}
The functor \(\mathrm{TFH}\) (Torsion-Free Higgs sheaves) associates to a scheme \(S\) the groupoid of pairs \((\mathcal{F}, \psi)\) where:
\begin{itemize}
\item \(\mathcal{F}\) is a flat family of torsion-free coherent sheaves of rank \(n\) and degree \(d\) on \(X_0 \times S\),
\item \(\psi \colon \mathcal{F} \to \mathcal{F} \otimes \omega_{X_0/S}\) is an \(\mathcal{O}_{X_0 \times S}\)-linear morphism (Higgs field).
\end{itemize}
\end{definition}

The natural transformation 
\[
\mathrm{GHB} \longrightarrow \mathrm{TFH}
\]
is given by pushforward along the contraction \(\pi \colon X_k \to X_0\):
\[
(\mathcal{E}, \Phi) \ \longmapsto \ (\pi_* \mathcal{E},\ \pi_* \Phi).
\]
This map contracts the exceptional components and produces a torsion-free Higgs sheaf on the nodal curve \(X_0\).

\subsubsection{Level-one functors}

We now restrict to the minimal (level-one) expanded degeneration \(X_1 = \widetilde{X}_0 \cup R\), where \(R \cong \mathbb{P}^1\) is a single exceptional rational curve attached at the node.

\begin{definition}
The functor \(\mathrm{GHB}^1\) associates to a scheme \(S\) the groupoid whose objects are triples \((\mathcal{E}, \Phi)\) consisting of
\begin{itemize}
\item a flat family of semistable Gieseker bundles \(\mathcal{E}\) on \(\mathfrak X\), where \(\mathfrak X \to X_0 \times S\) is a family of semistable modifications of \(X_0\) with at most one exceptional rational component (i.e., each fibre is either isomorphic to \(X_0\) or to the level-one modification \(X_1\)),
\item a Higgs field \(\Phi \colon \mathcal{E} \to \mathcal{E} \otimes \omega_{\mathfrak X/S}\) compatible with the Gieseker structure at the exceptional component (when present),
\end{itemize}
together with isomorphisms of such data.
\end{definition}

\begin{definition}
The functor \(\mathrm{TFH}^1\) associates to a scheme \(S\) the groupoid of pairs \((\mathcal{F}, \psi)\) where
\begin{itemize}
\item \(\mathcal{F}\) is a flat family of torsion-free coherent sheaves on \(X_0 \times S\) of rank \(n\) and degree \(d\),
\item \(\psi \colon \mathcal{F} \to \mathcal{F} \otimes \omega_{X_0\times S/S}\) is a Higgs field,
\item the local type of \(\mathcal{F}\) at the node is of the form either \(\mathcal O^{n}\) or \(\mathcal{O}^{n-1} \oplus m\).
\end{itemize}
\end{definition}

The natural map of functors
\[
\mathrm{GHB}^1 \longrightarrow \mathrm{TFH}^1
\]
is induced by pushforward along the contraction morphism \(\pi \colon \mathfrak X \to X_0\):
\[
(\mathcal{E}, \Phi) \ \longmapsto \ (\pi_* \mathcal{E},\ \pi_* \Phi).
\]
This map contracts the exceptional component (when present) and produces a torsion-free Higgs sheaf of the corresponding local type.

\

\begin{proposition}\label{fullcartesian}\label{hello}
The following square of functors is commutative:
\begin{equation}\label{sqfull}
\begin{tikzcd}
\mathrm{GHB} \arrow[r]\arrow[d] & \mathrm{TFH}\arrow[d] \\
\mathrm{GVB} \arrow[r] & \mathrm{TF}
\end{tikzcd}
\end{equation}
Moreover, the map $\mathrm{GHB}\to \mathrm{TFH}$ is surjective and therefore the moduli stack \(\mathrm{TFH}\) is irreducible because \(\mathrm{GHB}\) is irreducible.
\end{proposition}

\begin{proof}
The commutativity of the square is obvious from the definitions of the maps. The horizontal arrows are given by pushforward $\pi_*$ along the contraction morphism $\pi\colon \mathfrak{X}\to X_0\times S$, and the vertical arrows are the forgetful functors forgetting the Higgs field.

First of all, it is well known that the bottom horizontal map is surjective \cite[Proposition 10.1.]{Ka}. We will now show that, for any object $(\mathcal{F},\psi)\in\mathrm{TFH}$ with underlying torsion-free sheaf $\mathcal{F}$, if $\mathcal{E}$ is a Gieseker vector bundle on a Gieseker curve $X_k$ such that $\pi_*\mathcal{E}\cong\mathcal{F}$, the induced map on Hom-spaces
\[
\Hom(\mathcal{E},\mathcal{E}\otimes\pi^*\omega_{X_0})\footnote{Notice that $\pi^*\omega_{X_0}\cong \omega_{X_k}$}\ \longrightarrow\ 
\Hom(\pi_*\mathcal{E},\pi_*\mathcal{E}\otimes\omega_{X_0})
\]
is an isomorphism.

\textbf{Injectivity.} Suppose $\pi_*\phi=0$ for a Higgs field $\phi\colon\mathcal{E}\to\mathcal{E}\otimes\pi^*\omega_{X_0}$. Then the restriction of $\phi$ to the main component $\widetilde{X}_0$ vanishes. Since $\mathcal{E}$ is a Gieseker bundle, this forces $\phi=0$ on the entire $\mathfrak{X}$. Hence, the map is injective.

\textbf{Surjectivity.} First, we prove that $R^1\pi_*(\End(\mathcal{E})\otimes\pi^*\omega_{X_0})=0$. Let $R\subset X_k$ denote the exceptional rational tree attached to the node. By the construction of Gieseker bundles, the restriction $\mathcal E|_R$ is a vector bundle on the tree of $\mathbb{P}^1$'s such that on each irreducible component $R_i\cong\mathbb{P}^1$ of $R$, the restriction $$\mathcal End \mathcal E\otimes \pi^*\omega_{X_0}|_{R_i}$$ is a direct sum of the form $\mathcal O(-1)^{\oplus a}\oplus \mathcal O^{\oplus b}\oplus \mathcal O(1)^{\oplus c}$. 

Thus $H^1(R_i,(\End(\mathcal E)\otimes \pi^*\omega_{X_0})|_{R_i})=0$ for each component. Since $R$ is a tree, the cohomology on the whole tree vanishes in degree $1$ by successive applications of the Mayer--Vietoris sequence. Therefore it follows that $R^1\pi_*(\End(\mathcal{E})\otimes\pi^*\omega_{X_0})=0$.

Consequently,
\[
H^0(X_k,\End(\mathcal{E})\otimes\pi^*\omega_{X_0})\ \cong\ H^0(X_0,\pi_*(\End(\mathcal{E})\otimes\pi^*\omega_{X_0})).
\]

Next, the natural morphism of sheaves

\[
\pi_*(\End(\mathcal{E})\otimes\pi^*\omega_{X_0})\ \longrightarrow\ 
\mathcal{H}om(\pi_*\mathcal{E},\pi_*\mathcal{E}\otimes\omega_{X_0})
\]

is an isomorphism. This is verified locally near the node: let $U\subset X_0$ be a small affine neighbourhood of the node, with preimage $V\cup R$ where $V\subset\widetilde{X}_0$ is affine. Then on $U$, the sheaf $\mathcal{H}om(\pi_*\mathcal{E},\pi_*\mathcal{E}\otimes\omega_{X_0/S})$ fits into a short exact sequence
\[
0\ \to\ \mathcal{H}om(\mathcal{F},\mathcal{F}\otimes\omega_{X_0})(U)\ \to\ 
\Hom(E|_V,E|_V\otimes\omega_{\widetilde{X}}|_V)\ \oplus\ 
\Hom(E|_R,E|_R\otimes\omega_{X_k}|_{R})
\ \to\ T_1\oplus T_2\ \to\ 0,
\]
where $T_1,T_2$ are torsion sheaves supported at the attachment points that encode the gluing conditions. This sequence shows that the sections of the pushforward sheaf $\pi_*(\End(\mathcal{E})\otimes\pi^*\omega_{X_0})$ over $U$ are precisely the compatible local sections, hence the natural map is an isomorphism.

Combining these, we obtain
\[
H^0(X_k,\End(\mathcal{E})\otimes\pi^*\omega_{X_0})\ \cong\ 
H^0(X_0,\mathcal{H}om(\mathcal{F},\mathcal{F}\otimes\omega_{X_0})),
\]
which proves surjectivity. 

Now consider a torsion-free Hitchin pair $(\mathcal F, \phi: \mathcal F\to \mathcal F\otimes \omega_{X_0})$ on the nodal curve $X_0$. Choose any Gieseker vector bundle $\mathcal E$ on a suitable Gieseker curve $X_k$ such that $\pi_*\mathcal E\cong \mathcal F$. We have established that there is an isomorphism between the spaces of Higgs fields on $\mathcal E$ and $\mathcal F$:

\[
H^0(X_k,\End(\mathcal{E})\otimes\pi^*\omega_{X_k})\ \cong\ 
H^0(X_0,\mathcal{H}om(\mathcal{F},\mathcal{F}\otimes\omega_{X_0})),
\]

Now the given Higgs field $\phi: \mathcal F\to \mathcal F\otimes \omega_{X_0}$ is an element of the right hand side $\phi\in H^0(X_0,\mathcal{H}om(\mathcal{F},\mathcal{F}\otimes\omega_{X_0}))$. Via the isomorphism, we get a Higgs field $\Phi\in H^0(X_k,\End(\mathcal{E})\otimes\pi^*\omega_{X_k})$ such that $\pi_*\Phi=\phi$. Therefore, the map $\mathrm{GHB}\to \mathrm{TFH}$ is surjective.  

Since $\mathrm{GHB}$ is irreducible \cite[Theorem A.12]{BDP} and the functor $\mathrm{GHB}\to\mathrm{TFH}$ is a surjective map, hence $\mathrm{TFH}$ is irreducible.
\end{proof}

\smallskip

\begin{proposition}\label{fullcartesian1}\label{irr}
The following square of functors is commutative:
\begin{equation}\label{sqfull}
\begin{tikzcd}
\mathrm{GHBD} \arrow[r]\arrow[d] & \mathrm{GPH}\arrow[d] \\
\mathrm{GVBD} \arrow[r] & \mathrm{GPB}
\end{tikzcd}
\end{equation}
Moreover, the map $\mathrm{GHBD}\to \mathrm{GPH}$ is surjective and therefore the moduli stack \(\mathrm{GPH}\) is irreducible because \(\mathrm{GHBD}\) is irreducible.
\end{proposition}

\begin{proof}
    Recall from \cite[\S 5. Definition of Gieseker Hitchin pair data]{D} that an element of $\mathrm{GVBD}$ is a Gieseker vector bundle $(X_k, \mathcal E)$ with a choice of a marked node $s$ of the Gieseker curve $X_k$. The map $h_{\bullet}: \mathrm{GVBD}\to \mathrm{GPB}$ is given by $h_{\bullet}(X_k, \mathcal E, s)\mapsto (E:=h_*(\widetilde{\mathcal E}(-s^+-s^-))(x^++x^-), F(E)\subset E_{x^+}\oplus E_{x^-})$, where $\widetilde{\mathcal E}:=q^*_s(\mathcal E)$ and $q_s: \widetilde{X_k}\to X_k$ is the seminormalisation of $X_k$ at the point $s$, and $F(E): \Gamma_{A}\hookrightarrow \widetilde{\mathcal E}_{s^+}\oplus \widetilde{\mathcal E}_{s^-}\hookrightarrow E_{x^+}\oplus E_{x^-}$ is the composition of the two map. Here $A: \widetilde{\mathcal E}_{s^+}\to \widetilde{\mathcal E}_{s^-}$ is the descent datum/ Gluing isomorphism for the bundle $\mathcal E$ and $\Gamma_A$ denotes the graph of the isomorphism $A$. Similarly, an element of $\mathrm{GHBD}$ is a $4$-tuple $(X_k, \mathcal E, s, \phi)$, where $\phi$ is a Higgs field on $\mathcal E$. Notice that in this case the flag $F(E)$ will be respected/ preserved by the naturally induced Higgs field $h_{\bullet}$ on $E$ (see \cite[\S 7.1 The GPH corresponding to a given Gieseker Hitchin pair data]{D}). 
    
      Now from the previous proposition \ref{hello}, we see that the space of the Higgs field $H^0(X_k, \End \mathcal E\otimes \omega_{X_k})=H^0(X_0, \mathcal Hom(\mathcal F, \mathcal F\otimes \omega_{X_0}))=\mathrm Hom(\mathcal F, \mathcal F\otimes \omega_{X_0})$, the space of the Higgs fields on the induced torsion free sheaf $\mathcal F$ on the original nodal curve $X_0$. But $\mathrm Hom(\mathcal F, \mathcal F\otimes \omega_{X_0})=\mathrm Hom_{F(E)}(E, E\otimes \Omega^1_{\widetilde{X_0}}(x^++x^-))$, the space of Higgs fields on a GPB $(E, F(E))$ which also induces the same torsion-free sheaf $\mathcal F$. Therefore, we have $H^0(X_k, \End \mathcal E\otimes \omega_{X_k})=\mathrm Hom_{F(E)}(E, E\otimes \Omega^1_{\widetilde{X_0}}(x^++x^-))$. Notice that the left hand side is precisely the fibre ($\mathbb C$-valued) of the vertical arrow on the left and the right hand side is precisely the fibre ($\mathbb C$-valued) of the vertical map on the right. 
      
      Now take any GPH $(E, F(E), \tilde \phi: E\to E\otimes\Omega^1_{\widetilde{X_0}}(x^++x^-))$. Since the map $\mathrm{GVBD}\to \mathrm{GPB}$ is surjective, choose any GVBD $(X_k, \mathcal E, s)$ which induces the GPB $(E, F(E))$. Since $H^0(X_k, \End \mathcal E\otimes \omega_{X_k})=\mathrm Hom_{F(E)}(E, E\otimes \Omega^1_{\widetilde{X_0}}(x^++x^-))$ and $\tilde \phi\in \mathrm Hom_{F(E)}(E, E\otimes \Omega^1_{\widetilde{X_0}}(x^++x^-))$, we get a unique $\tilde \Phi\in H^0(X_k, \End \mathcal E\otimes \omega_{X_k})$ which induces the GPH $(E, F(E), \tilde \phi)$ under the map $h_{\bullet}$. Therefore the map $\mathrm{GHBD}\to \mathrm{GPH}$ is surjective. Also note that $\mathrm{GHBD}$ is the normalisation of $\mathrm{GHB}$. Therefore since $\mathrm{GHB}$ is irreducible \cite[Theorem A.12]{BDP}, $\mathrm {GHBD}$ is also irreducible. Since the map $\mathrm{GHBD}\to \mathrm{GPH}$ is surjective and $\mathrm{GHBD}$ is irreducible therefore $\mathrm{GPH}$ is also irreducible. This concludes the proof. 
\end{proof}

\

Bhosle's moduli space of generalized parabolic bundles, and its Hitchin-pair analogue $\mathrm M_{\mathrm{GPH}}$, are constructed using a stability condition that depends on an auxiliary real parameter $\alpha$. We recall the definition here.

\begin{definition}[$\alpha$-semistability of a GPB]\label{def:alpha-GPB}
Let $(E,F(E))$ be a generalized parabolic bundle of rank $n$ and degree $d$ on $\widetilde X_0$, and let $\alpha \in [0,1]$. For a subbundle $E' \subseteq E$ of rank $n'$ and degree $d'$, set
\[
\delta(E') \ \colonequals\ \dim\bigl(F(E) \cap E'_D\bigr), \qquad E'_D \colonequals E'_{x^+}\oplus E'_{x^-} \ \subseteq\ E_D,
\]
and define the \emph{$\alpha$-slope}
\[
\mu_\alpha(E') \ \colonequals\ \frac{d' + \alpha\,\delta(E')}{n'}.
\]
The pair $(E,F(E))$ is \emph{$\alpha$-semistable} (resp.\ \emph{$\alpha$-stable}) if
\[
\mu_\alpha(E') \ \le\ \mu_\alpha(E) \qquad \bigl(\text{resp.}\ \mu_\alpha(E') < \mu_\alpha(E)\bigr)
\]
for every proper subbundle $0 \neq E' \subsetneq E$, where $\mu_\alpha(E) = \dfrac{d+\alpha n}{n}$.
\end{definition}

\begin{definition}[$\alpha$-semistability of a GPH]\label{def:alpha-GPH}
A generalized parabolic Hitchin pair $(E,\phi,F(E))$ is \emph{$\alpha$-semistable} (resp.\ \emph{$\alpha$-stable}) if the inequality of \cref{def:alpha-GPB} holds for every proper $\phi$-invariant subbundle $E' \subsetneq E$, i.e.\ every $E'$ with $\phi(E') \subseteq E' \otimes \Omega^1_{\widetilde X_0}(x^++x^-)$:
\[
\mu_\alpha(E') \ \le\ \mu_\alpha(E) \qquad \bigl(\text{resp.}\ <\bigr).
\]
\end{definition}

\begin{remark}[Special values of $\alpha$]\label{rem:alpha-special}
\begin{enumerate}
\item At $\alpha = 0$, \cref{def:alpha-GPH} reduces to ordinary slope-stability of the underlying Hitchin pair $(E,\phi)$ on $\widetilde X_0$, independently of the flag $F(E)$.
\item At $\alpha = 1$, $\alpha$-stability of $(E,\phi,F(E))$ is equivalent to stability of the associated torsion-free Higgs sheaf $(\mathcal F,\phi_{\mathcal F})$ on $X_0$.
\end{enumerate}
\end{remark}

\begin{remark}[Chamber structure]\label{rem:alpha-walls}
For fixed $(n,d)$ there are only finitely many rational values of $\alpha \in [0,1]$ (\emph{walls}) at which $\alpha$-semistability fails to coincide with $\alpha$-stability, or at which the set of $\alpha$-semistable GPHs changes; on each open interval between consecutive walls (a \emph{chamber}) the moduli space $\mathrm M_{\mathrm{GPH}}$ is independent of $\alpha$. Throughout this paper, ``stable'' GPH means $\alpha$-stable for a fixed generic $\alpha$ in such a chamber; since $\gcd(n,d)=1$, generic $\alpha$-semistability coincides with $\alpha$-stability, so $\mathcal M_{\mathrm{GPH}}$ is a fine moduli space away from wall-crossing behaviour.
\end{remark}

\begin{remark}[Reduction to $\alpha=1$]\label{rem:alpha-one}
Throughout this and the following sections we assume $\alpha = 1$ and $\gcd(n,d)=1$. We point out that if the moduli space is irreducible at $\alpha=1$, then it is irreducible for every $\alpha$: the moduli spaces $\mathcal M_{\mathrm{GPH}}^{\alpha}$ for $\alpha$ ranging over different chambers (\cref{rem:alpha-walls}) are related to one another by standard birational transformations known as \emph{flip-flops}, occurring at each wall; since a flip-flop is a birational modification supported on a proper closed subset, it does not change the number of irreducible components. Hence it suffices to establish irreducibility for the single value $\alpha=1$, which is the case treated in \cref{irr}.

Moreover, since the moduli \emph{stack} $\mathrm M_{\mathrm{GPH}}$ is irreducible (it does not involve any stability parameter) (\cref{irr}), so is its coarse moduli space $\mathcal M_{\mathrm{GPH}}$: the natural map $\mathrm M_{\mathrm{GPH}} \to \mathcal M_{\mathrm{GPH}}$ is surjective and continuous (in the Zariski topology), and the continuous image of an irreducible topological space is irreducible. Thus $\mathcal M_{\mathrm{GPH}}$ is irreducible for $\alpha=1$, and by the flip-flop remark above, for every $\alpha$.
\end{remark}

\begin{remark}[Convention]\label{rem:stable-convention}
\textbf{From this point onward, unless stated otherwise, by a stable generalized parabolic Hitchin pair we shall always mean an $\alpha$-stable one with $\alpha = 1$ in the sense of \cref{def:alpha-GPH}, and similarly for semistability. We drop the prefix ``$\alpha$-'' and the subscript accordingly.}
\end{remark} 

We will now describe a codimension $1$ locus inside the coarse moduli space $\mathcal M_{\mathrm{GPH}}:$  

\begin{corollary}\label{codim3}
Let \(\mathcal{M}^{1}_{GPH} \subset \mathcal{M}_{\mathrm{GPH}}\) be the locus consisting of points where the generalized parabolic structure \(A: E_{x^+} \to E_{x^-}\) satisfies \(\operatorname{rank} A \geq n-1\). Then
\[
\operatorname{codim}_{\mathcal{M}_{\mathrm{GPH}}} \bigl( \mathcal{M}_{\mathrm{GPH}} \setminus \mathcal{M}^{1}_{GPH} \bigr) \geq 2.
\]
\end{corollary}

\begin{proof}
Let \(\widetilde{X} = \widetilde{\mathcal{M}}_{\mathrm{GHB}}\) be the normalization of the moduli space of stable Gieseker--Higgs bundles (which is irreducible by \cite[Theorem A.12]{BDP}), $X:=\mathcal M_{GHB}$ the moduli space of stable Gieseker-Higgs bundles and let \(Y = \mathcal{M}_{\mathrm{GPH}}\) be the moduli space of stable $\Omega^1_{\widetilde{X_0}}(x^++x^-)$-twisted generalized parabolic Hitchin pairs. 

There is a natural proper morphism 
\[
f=h_{\bullet} \colon \widetilde{X} \longrightarrow Y
\]
(see \cite{D}). This morphism is surjective onto \(Y\) (see Proposition \ref{irr}) and birational, since it is an isomorphism over the open locus where \(\operatorname{rank} A = n\) and both $\widetilde X$ and $Y$ are irreducible.

Let \(\widetilde{X^1} \subset \widetilde{X}\) be the open subset consisting of Gieseker--Higgs bundle data only on the modifications $X_0$ and $X_1$ of the nodal curve (see def. \ref{mod11}). Notice that the Gieseker-Higgs bundle data in $\widetilde{X^1}$ are precisely the ones that correspond to GPHs of the form $(E, \phi, A)$ such that \(\operatorname{rank} A \geq n-1\) (see \cite[Proposition 10.1.]{Ka}). Since the moduli space $\widetilde{X}=\widetilde{\mathcal M_{GHB}}$ is irreducible and $\mathcal M_{GHB}$ is formally smooth over the space of expanded degenerations of the nodal curve, it follows that 
\[
\operatorname{codim}_{\widetilde{X}} (\widetilde{X} \setminus \widetilde{X^1}) \geq 2.
\]

This is due to the fact that the singular locus $Sing(X)=X\setminus X^0$ and $Sing(Sing(X))=X\setminus X^1$ (see \cite[Section 6]{DI}). Notice that $X\setminus X^1$ is the locus consisting of Gieseker-Higgs bundles on the expanded degenerations where the number of $\mathbb P^1$'s are greater than or equal to $2$. 

Let us denote by \(Y^1 := f(\widetilde{X^1}) = \mathcal{M}^{1}_{GPH}\). Then it is easy to see that \(f\) is restricted to a bijective morphism \(\widetilde{X^1} \rightarrow  Y^1\), and that \(Y \setminus Y^1 = f(\widetilde X \setminus \widetilde{X^1})\). 

Since \(f\) is proper and surjective, and \(\operatorname{codim}_{\widetilde X} (\widetilde X \setminus \widetilde{X^1}) \geq 2\), we have
\[
\operatorname{codim}_Y (Y \setminus Y^1) \geq 2.
\]
This completes the proof.
\end{proof}

\section{\textbf{Local complete Intersection Property}}

In this section, we prove that the moduli space of generalized parabolic Hitchin pairs \(\mathcal{M}_{\mathrm{GPH}}\) is a local complete intersection. The strategy is to realize \(\mathcal{M}_{\mathrm{GPH}}\) as the zero locus of a natural section of a vector bundle over a fibre product of well-understood moduli stacks and that it has exactly the expected dimension. We chose to use the language of stacks for convenience. The reader could also use the language of Quot schemes for this purpose which can be thought of as an Atlas for the Artin stacks. 

\begin{theorem}\label{thm:LCI}
The moduli space \(\mathcal{M}_{\mathrm{GPH}}\) is a local complete intersection.
\end{theorem}

\begin{proof}
Consider the following fibre product diagram of moduli stacks:
\[
\begin{tikzcd}
\mathrm{M}_{\mathrm{GPH}} \arrow{r} & 
\mathrm{M}_{\mathrm{HP}}\times_{\mathrm{M}_{\mathrm{Bun}}} \mathrm{M}_{\mathrm{GPB}}
\arrow{r} \arrow{d} & 
\mathrm{M}_{\mathrm{GPB}} \arrow{d} \\
& \mathrm{M}_{\mathrm{HP}} \arrow{r} & \mathrm{M}_{\mathrm{Bun}}
\end{tikzcd}
\]
where 
\begin{itemize}
\item \(\mathrm M_{\mathrm{Bun}}\) is the moduli stack of vector bundles of rank \(n\) and degree \(d\) on \(\widetilde{X}_0\),
\item \(\mathrm M_{\mathrm{HP}}\) is the moduli stack of Hitchin pairs \((E,\phi)\) of rank \(n\) and degree \(d\) on \(\widetilde{X}_0\),
\item \(\mathrm M_{\mathrm{GPB}}\) is the moduli stack of generalised parabolic bundles of rank \(n\) and degree \(d\) on \(\widetilde{X}_0\) (with parabolic structure at \(\{x^+,x^-\}\)),
\item \(\mathrm M_{\mathrm{HP}} \times_{\mathrm M_{\mathrm{Bun}}} \mathrm M_{\mathrm{GPB}}\) parametrises triples \(((E,\phi), (E,F(E)))\) with the same underlying vector bundle \(E\).
\end{itemize}

There exists a universal vector bundle \(\mathcal{E}\) on \(\widetilde{X}_0 \times \mathrm M_{\mathrm{HP}}\) together with a universal Higgs field
\[
\phi \colon \mathcal{E} \to \mathcal{E} \otimes \Omega_{\widetilde{X}_0}(x^+ + x^-).
\]
There is also a universal flag
\[
F(\mathcal{E}) \subset \mathcal{E}_{x^+} \oplus \mathcal{E}_{x^-}
\]
on \(\widetilde{X}_0 \times \mathrm M_{\mathrm{GPB}}\).

On the fibre product \(\mathrm M_{\mathrm{HP}} \times_{\mathrm M_{\mathrm{Bun}} \mathrm M_{\mathrm{GPB}}}\), we obtain a natural morphism of bundles
\[
F(\mathcal{E}) \;\longrightarrow\; \mathcal{E}_{x^+} \oplus \mathcal{E}_{x^-} \;\longrightarrow\;
\bigl(\mathcal{E}_{x^+} \oplus \mathcal{E}_{x^-}\bigr) \otimes \Omega_{\widetilde{X}_0}(x^++x^-)
\;\longrightarrow\;
\frac{\mathcal{E}_{x^+} \oplus \mathcal{E}_{x^-}}{F(\mathcal{E})} \otimes \Omega_{\widetilde{X}_0}(x^++x^-).
\]
The vanishing locus of this global section \(\sigma\) (i.e., the locus where the Higgs field preserves the generalised parabolic structure) is precisely the moduli stack \(\mathrm{M}_{\mathrm{GPH}}\).

The ambient stack \(\mathrm M_{\mathrm{HP}} \times_{\mathrm M_{\mathrm{Bun}} \mathrm M_{\mathrm{GPB}}}\) is a local complete intersection: the moduli stack of Hitchin pairs \(\mathrm M_{\mathrm{HP}}\) is known to be LCI (see \cite{BDP,Nit}), \(\mathrm M_{\mathrm{GPB}}\) is smooth, and the fibre product over \(\mathrm M_{\mathrm{Bun}}\) preserves the LCI property since the forgetful morphism \(\mathrm M_{\mathrm{GPB}} \to \mathrm M_{\mathrm{Bun}}\) is smooth.

Now from Proposition \ref{irr} we know that \(\mathcal{M}_{\mathrm{GPH}}\) is irreducible. The expected dimension of \(\mathcal{M}_{\mathrm{GPH}}\) is
\[
\dim \mathcal{M}_{\mathrm{GPH}} = 2n^2(g-1) + 2.
\]
The ambient fibre product has dimension
\[
\dim(\mathrm M_{\mathrm{HP}} \times_{\mathrm{M}_{\mathrm{Bun}}} \mathrm{M}_{GPB}) = 2n^2(g-1) + 2 + n^2 = 2n^2(g-1) + n^2 + 2.
\]

Thus \(\mathrm{M}_{\mathrm{GPH}}\) is cut by a section of a vector bundle of rank \(n^2\) on an LCI ambient stack, and has the expected dimension. Therefore \(\mathrm{M}_{\mathrm{GPH}}\) is LCI.

Finally, since \(\gcd(n,d)=1\), the stabilizers at stable points are uniformly $\mathbb C^*$, the map $\mathrm M^{st}_{\mathrm{GPH}}\to \mathcal M_{\mathrm{GPH}}$ is smooth. Since $\mathrm M^{st}_{\mathrm{GPH}}$ is an LCI, the coarse moduli space \(\mathcal{M}_{\mathrm{GPH}}\) is also an LCI. This completes the proof. 
\end{proof}

\section{\textbf{Normality and Applications}}

In this section we apply the deformation theory developed earlier to study the geometry of the coarse moduli space \(\mathcal{M}_{\mathrm{GPH}}\) of stable generalized parabolic Hitchin pairs in the case where $g.c.d(n,d)=1$. 

First of all, notice that there exists a proper morphism 
\[
h_{\bullet} \colon \mathcal{M}_{\mathrm{GHB}} \longrightarrow \mathcal{M}_{\mathrm{GPH}}
\]
from the moduli space of Gieseker--Higgs bundles to the moduli space of generalized parabolic Hitchin pairs \cite{D}. The map \(h_{\bullet}\) restricts to an isomorphism over the open locus in \(\mathcal{M}_{\mathrm{GPH}}\) where \(\operatorname{rank} A = n\).

\begin{theorem}\label{NT}
The variety \(\mathcal{M}_{\mathrm{GPH}}\) is a local complete intersection that is smooth in codimension \(1\), and therefore is a normal variety.
\end{theorem}

\begin{proof}
By Theorem~\ref{thm:LCI}, the moduli space $\mathcal {M}_{\mathrm{GPH}}$ is a local complete intersection.

It is therefore enough to show that \(\mathcal{M}_{\mathrm{GPH}}\) is smooth in codimension one. Let \(\mathcal{M}^{1}_{GPH}\) be the open locus consisting of points where the generalized parabolic structure is of the form $(E, \phi, A)$ with \(\operatorname{rank} A \geq n-1\). From Proposition \ref{irr}, it follows that \(\mathcal{M}_{\mathrm{GPH}}\) is irreducible. Moreover, from Corollary \ref{codim3} we know that the complement 
\[
\mathcal{M}_{\mathrm{GPH}} \setminus \mathcal{M}^{1}_{\mathrm{GPH}}
\]
has codimension at least \(2\).

Split it into two parts:
\[
\mathcal{M}^{1}_{\mathrm{GPH}} = \mathcal{M}^{0}_{\mathrm{GPH}} \ \sqcup \ (\mathcal{M}^{1}_{\mathrm{GPH}} \setminus \mathcal{M}^{0}_{\mathrm{GPH}}),
\]
where \(\mathcal{M}^{0}_{\mathrm{GPH}}\) is the open locus where \(\operatorname{rank} A = n\) (full rank gluing). Let us denote by $\mathcal M^{0}_{\mathrm{GHB}}$ the open locus (inside $\mathcal M_{\mathrm{GHB}}$) of stable Gieseker-Higgs bundles on the original nodal curve $X_0$ and by $\mathcal M^{0}_{\mathrm{TFH}}$ the open locus (inside $\mathcal M_{\mathrm{TFH}}$) of stable Higgs bundles on the nodal curve $X_0$. 

From the definition it is clear that $\mathcal M^{0}_{\mathrm{GHB}}\cong \mathcal M^{0}_{\mathrm{TFH}}\cong \mathcal{M}^{0}_{\mathrm{GPH}}$. Since $\mathcal M^{0}_{\mathrm{GHB}}$ is smooth, \(\mathcal{M}^0_{\mathrm{GPH}}\) is smooth.

On the locus \(\mathcal{M}^{1}_{\mathrm{GPH}} \setminus \mathcal{M}^{0}_{\mathrm{GPH}}\) (i.e., where \(\operatorname{rank} A = n-1\)), consider the generic point \((E_{\bullet},\phi)\). By Lemma~\ref{hypecomp}, we have
  \[
  \dim \mathbb{H}^2(\mathcal{C}_{(E_{\bullet},\phi)}) = 1.
  \]
  Therefore the generic codimension $1$ point is smooth in $\mathcal{M}_{\mathrm{GPH}}$, because $\dim \mathbb{H}^2(\mathcal{C}_{(E_{\bullet},\phi)}) = 1$ is the generic value (see Lemma \ref{H2:Vect}). In other words, it is smooth in codimension $1$. A local complete intersection variety that is smooth in codimension one is normal (by Serre's criterion: it satisfies \(R_1\) and \(S_2\)). This completes the proof.
\end{proof}

\subsection{Applications}

\begin{corollary}\label{cor:log-symplectic}
The open locus \(\mathcal{M}^{1}_{\mathrm{GPH}} \subset \mathcal{M}_{\mathrm{GPH}}\) consisting of points where \(\operatorname{rank} A \geq n-1\) carries a natural log-symplectic form.

More precisely: the map $h_\bullet\colon\widetilde{\mathcal M}^1_{\mathrm{GHB}}\to\mathcal M^1_{\mathrm{GPH}}$ is proper and bijective, and by Theorem~\ref{NT}, $\mathcal M_{\mathrm{GPH}}$ --- and hence its open subvariety $\mathcal M^1_{\mathrm{GPH}}$ --- is normal. A proper, bijective morphism onto a normal variety is an isomorphism, by \emph{Zariski's Main Theorem}. Applying this to $h_\bullet$ gives a natural isomorphism of schemes
\[
\mathcal{M}^{1}_{\mathrm{GPH}} \ \cong \ \widetilde{\mathcal{M}}^{1}_{\mathrm{GHB}},
\]
where \(\widetilde{\mathcal{M}}^{1}_{\mathrm{GHB}}\) denotes the normalization of the moduli space of level-one Gieseker--Higgs bundles (or its coarse space).

(This is also why, in the proof of Corollary~\ref{codim3}, we were careful to establish only that $\widetilde X^1\to Y^1$ is \emph{bijective} rather than an isomorphism: at that point in the paper $\mathcal M_{\mathrm{GPH}}$ had not yet been shown to be normal, so Zariski's Main Theorem was not yet available to upgrade bijectivity to an isomorphism. It is only here, after Theorem~\ref{NT}, that the same bijective morphism can be promoted to an isomorphism.)

Since $\widetilde{\mathcal{M}}^{1}_{\mathrm{GHB}}$ admits a natural log-symplectic structure (\cite{DI}), the isomorphism above transports this to a log-symplectic form on \(\mathcal{M}^{1}_{\mathrm{GPH}}\).

Finally, as \(\mathcal{M}_{\mathrm{GPH}}\) is a normal variety (Theorem~\ref{NT}) and $\operatorname{codim}_{\mathcal M_{\mathrm{GPH}}}(\mathcal{M}_{\mathrm{GPH}} \setminus \mathcal{M}^{1}_{\mathrm{GPH}}) \geq 2$ (Corollary~\ref{codim3}), Hartogs' extension theorem --- a section of a locally free sheaf on a normal variety, defined on the complement of a closed subset of codimension $\geq 2$, extends uniquely across it --- shows that this log-symplectic form, viewed as a section of $\wedge^2 T_{\mathcal{M}_{\mathrm{GPH}}}$ over $\mathcal M^1_{\mathrm{GPH}}$, extends uniquely to a Poisson bivector on the whole moduli space \(\mathcal{M}_{\mathrm{GPH}}\).
\end{corollary}

\appendix

\section{Computation tool for hypercohomology}

In this appendix, we collect two useful computational tools that are repeatedly used in the hypercohomology calculations throughout the paper. The spectral sequence approach, and studying the short exact sequence of complexes provides the main technical machinery for all obstruction computations in this paper.

\subsection{Computation tool I: Via the hypercohomology spectral sequence}\label{Append}

An important computational tool is the spectral sequence arising from the deformation complex itself. Recall that
\[
\mathcal{C}^\bullet = \bigl[ \mathcal{C}^0 \to \mathcal{C}^1 \to \mathcal{C}^2 \bigr],
\]
where
\[
\mathcal{C}^0 = \pi_*\End(E), \qquad
\mathcal{C}^1 = \pi_*\End(E)\otimes \omega_{X_0} \oplus \Hom(F,Q), \qquad
\mathcal{C}^2 = \Hom(F, Q ),
\]
and \(Q\) is the quotient sheaf associated with the generalized parabolic structure at the node.

Since \(X_0\) is a projective curve, we have \(H^q(X_0, -) = 0\) for all \(q \geq 2\). The hypercohomology spectral sequence associated with \(\mathcal{C}^\bullet\) therefore takes the form
\[
E_1^{p,q} = H^q(X_0, \mathcal{C}^p) \ \Longrightarrow \ \mathbb{H}^{p+q}(X_0, \mathcal{C}^\bullet),
\]
and has only two nonzero rows (\(q=0\) and \(q=1\)). The relevant portion of the \(E_1\)-page is
\[
\begin{array}{c|ccc}
q \backslash p & 0 & 1 & 2 \\
\hline
1 & H^1(\mathcal{C}^0) & H^1(\mathcal{C}^1) & H^1(\mathcal{C}^2) \\
0 & H^0(\mathcal{C}^0) & H^0(\mathcal{C}^1) & H^0(\mathcal{C}^2)
\end{array}
\]

The differentials \(d_1\) act horizontally. For the computation of \(\mathbb{H}^2(\mathcal{C}^\bullet)\), the important terms on the \(E_2\)-page are \(E_2^{2,0}\) and \(E_2^{1,1}\), and we have
\[
\mathbb{H}^2(\mathcal{C}^\bullet) \ \simeq \ E_2^{2,0} \ \oplus \ E_2^{1,1},
\]
where
\[
E_2^{2,0} = coker\bigl( H^0(\mathcal{C}^1) \to H^0(\mathcal{C}^2) \bigr)
\]
is the cokernel of the map between global sections in degree zero, and \(E_2^{1,1}\) is the cohomology at position \((1,1)\) on the \(E_2\)-page.

Given a GPH \((E_{\bullet},\phi)\) whose underlying vector itself is stable we see that  \(E_2^{1,1} = 0\) because 
\begin{equation}\label{E-computation}
E_2^{1,1} \ \cong \ \frac{H^1(X_0, \pi_*\End(E)\otimes \omega_{X_0})}{\operatorname{im}\bigl( H^1(\End(E)) \xrightarrow{[\cdot,\phi]} H^1(\End(E)\otimes \omega_{X_0}) \bigr)} 
\end{equation}

But $H^1(X_0, \pi_*\End(E)\otimes \omega_{X_0})=H^1(\tilde{X_0}, \End(E)\otimes \Omega^1_{\tilde{X_0}}(x^++x^-))\cong H^0(\widetilde{X_0}, \End E\otimes \mathcal O(-x^+-x^-))^{\vee}$.

Because $E$ is stable, $H^0(\widetilde{X_0}, \End E\otimes \mathcal O(-x^+-x^-))=0$. Therefore, $E_2^{1,1}=0$. 

Therefore, 
\[
\dim \mathbb{H}^2(\mathcal{C}_{(E_{\bullet},\phi)}) = \dim E_2^{2,0} = \dim coker\bigl( H^0(\mathcal{C}^1) \to H^0(\mathcal{C}^2) \bigr).
\]
The main task thus reduces to understanding this cokernel, which is governed by the local behavior of the differential at the node together with the global lifting properties coming from the residue map on normalization \(\widetilde{X}_0\).

\subsection{Computation tool II: Via a short exact sequence of complexes}

We frequently make use of the following natural short exact sequence of complexes of sheaves on the nodal curve \(X_0\):
\[
0 \longrightarrow B_{\bullet}[-1] \longrightarrow \mathcal{C}_{(E_{\bullet},\phi)} \longrightarrow \mathcal{D}_{\bullet} \longrightarrow 0,
\]
where
\begin{enumerate}
\item \(B_{\bullet} = \bigl[0 \to \Hom(F(E), E_D/F(E)) \to \Hom(F(E), (E_D/F(E))) \to 0\bigr]\)
is the \emph{parabolic correction complex},

\item \(\mathcal{C}_{(E_{\bullet},\phi)} = \bigl[ \pi_*\End(E) \to \pi_*(\End(E)\otimes \omega_{X_0}) \oplus \Hom(F(E), E_D/F(E))
\to \Hom(F(E), (E_D/F(E))) \bigr]\)
is the full deformation complex of the GPH,

\item \(\mathcal{D}_{\bullet} = \bigl[ \pi_*\End(E) \to \pi_*(\End(E)\otimes \omega_{X_0}) \to 0 \bigr]\)
is the usual deformation complex of the underlying Higgs bundle \((E,\phi)\).
\end{enumerate}

This short exact sequence induces a long exact sequence in hypercohomology:
\[
\begin{tikzcd}[column sep=tiny, row sep=tiny]
0 \rar & \mathbb{H}^0(B_{\bullet}[-1]) \rar & \mathbb{H}^0(\mathcal{C}) \rar & \mathbb{H}^0(\mathcal{D}) \rar
& \mathbb{H}^1(B_{\bullet}[-1]) \rar & \mathbb{H}^1(\mathcal{C}) \rar & \mathbb{H}^1(\mathcal{D}) \\
\rar & \mathbb{H}^2(B_{\bullet}[-1]) \rar & \mathbb{H}^2(\mathcal{C}) \rar & \mathbb{H}^2(\mathcal{D}) \rar
& \mathbb{H}^3(B_{\bullet}[-1]) \rar & \mathbb{H}^3(\mathcal{C}) \rar & \mathbb{H}^3(\mathcal{D}) \rar & \cdots
\end{tikzcd}
\]

Since the sheaves in \(B_{\bullet}\) are supported at the node (and hence acyclic away from it), the hypercohomology of \(B_{\bullet}\) coincides with the ordinary cohomology of its global sections complex. Explicitly, we have:
\begin{align*}
\mathbb{H}^0(B_{\bullet}[-1]) &= 0, \\
\mathbb{H}^1(B_{\bullet}[-1]) &= H^0\bigl(\Hom(F(E), E_D/F(E))\bigr)^{\phi}
= \bigl\{\psi \in \Hom(F(E), E_D/F(E)) \ \big|\ [\psi, \phi] = 0 \bigr\}, \\
\mathbb{H}^2(B_{\bullet}[-1]) &= \frac{\Hom(F(E), (E_D/F(E))\otimes \omega_{X_0})}{\{[\psi,\phi] \mid \psi \in \Hom(F(E), E_D/F(E))\}}, \\
\mathbb{H}^i(B_{\bullet}[-1]) &= 0 \quad \forall\, i \geq 3.
\end{align*}

By Serre duality for the Hitchin complex \(\mathcal{D}_{\bullet}\), we also have
\[
\mathbb{H}^3(\mathcal{D}_{\bullet}) = 0, \qquad
\mathbb{H}^2(\mathcal{D}_{\bullet}) = \Hom\bigl((E,\phi), (E(-x^+-x^-),\phi)\bigr).
\]

A direct Čech computation shows that
\[
\mathbb{H}^0(\mathcal{C}_{(E_{\bullet},\phi)}) = \bigl\{\psi \in \Aut(E,\phi) \ \big|\ (\pi_*\psi)(F(E)) \subseteq F(E) \bigr\}.
\]
When the GPH \((E_{\bullet},\phi)\) is stable, this reduces to the scalars, so
\[
\mathbb{H}^0(\mathcal{C}_{(E_{\bullet},\phi)}) \ \cong\ \mathbb{C}.
\]

Substituting these identifications into the long exact sequence yields
\[
0 \to \mathbb{H}^0(\mathcal{C}) \to \mathbb{H}^0(\mathcal{D}) \to H^0(B_{\bullet}) \to
\mathbb{H}^1(\mathcal{C}) \to \mathbb{H}^1(\mathcal{D}) \to H^1(B_{\bullet}) \to
\mathbb{H}^2(\mathcal{C}) \to \mathbb{H}^2(\mathcal{D}) \to 0.
\]

\begin{remark}[Stable underlying Higgs bundle]
If the underlying Higgs bundle \((E,\phi)\) is stable, then \(\mathbb{H}^0(\mathcal{D}_{\bullet}) \cong \mathbb{C}\) and \(\mathbb{H}^2(\mathcal{D}_{\bullet}) = 0\). The long exact sequence then reduces further to
\[
0 \to H^0(B_{\bullet}) \to \mathbb{H}^1(\mathcal{C}_{(E_{\bullet},\phi)}) \to \mathbb{H}^1(\mathcal{D}_{\bullet}) \to
H^1(B_{\bullet}) \to \mathbb{H}^2(\mathcal{C}_{(E_{\bullet},\phi)}) \to 0.
\]
Moreover, \(\dim \mathbb{H}^1(\mathcal{D}_{\bullet})\) is constant on each connected component of the moduli space of stable Higgs bundles.
\end{remark}

\begin{remark}[Connecting homomorphism] \label{rem:connecting-map}
When the flag is concentrated at one branch, i.e., \(F(E) = E_{x^+}\) or \(F(E) = E_{x^-}\), the connecting homomorphism \(\mathbb{H}^1(\mathcal{D}_{\bullet}) \to H^1(B_{\bullet})\) vanishes. In this case we have a short exact sequence
\[
0 \to H^0(B_{\bullet}) \to \mathbb{H}^1(\mathcal{C}_{(E_{\bullet},\phi)}) \to \mathbb{H}^1(\mathcal{D}_{\bullet}) \to 0,
\]
and an isomorphism
\[
H^1(B_{\bullet}) \ \cong\ \mathbb{H}^2(\mathcal{C}_{(E_{\bullet},\phi)}).
\]
\end{remark}

\

\bibliography{biblio}
\bibliographystyle{amsalpha}

\end{document}